\documentclass{amsart}

\usepackage[margin=2.6cm]{geometry}
\usepackage{amsmath,amssymb,amsthm,mathrsfs}
\usepackage{graphicx,color,enumitem,tikz}
\usepackage{appendix}
\usepackage{cancel}
\usepackage{ulem}
\usepackage[colorlinks=true,pdfborder={0 0 0}]{hyperref}
\hypersetup{urlcolor=blue,citecolor=red,linkcolor=blue}
\usepackage{etoolbox}
\newcounter{taggedeq}
\pretocmd{\equation}{\stepcounter{taggedeq}}{}{}

\newtheorem{theo}{Theorem}
\newtheorem{prop}[theo]{Proposition}
\newtheorem{proposition}[theo]{Proposition}
\newtheorem{lem}[theo]{Lemma}

\newtheorem*{conj}{Conjecture}
\theoremstyle{remark}
\newtheorem{rem}[theo]{Remark}

1
\makeatletter\@namedef{subjclassname@2020}{\textup{2020}
  Mathematics Subject Classification}\makeatother

\newcommand\N{{\mathbb N}}
\newcommand\R{{\mathbb R}}

\renewcommand\P{{\mathbb P}}
\newcommand{\dd}{\,\mathrm d}
\newcommand{\antiSigma}{A}
\newcommand{\wgt}{\lfloor\nabla\phi\rceil}

\newcommand{\bn}{{\bf n}}
\newcommand{\bu}{{\bf u}}

\renewcommand{\em}{\sl}
\def\MM{{\mathcal M}}
\def\RR{{\mathcal R}}

\def\nablaD{D}
\def\nablasym{D^s}
\def\nablaskew{D^a}

\def\set#1{\left\{#1\right\}}
\def\D{\partial}

\def\CRD{C_{\mathrm{\scriptscriptstyle RD}}}

\def\CSPK{C_{\mathrm{\scriptscriptstyle SPK}}}

\def\CRV{C_{\mathrm{\scriptscriptstyle RV}}}
\def\CRVL{C_{\mathrm{\scriptscriptstyle RVL}}}
\def\CRD{C_{\mathrm{\scriptscriptstyle RD}}}

\def\CRVZ{C_{\mathrm{\scriptscriptstyle RV0}}}
\def\CK{C_{\mathrm{\scriptscriptstyle K}}}
\def\CKZ{C_{\mathrm{\scriptscriptstyle K0}}}
\def\CPK{C_{\mathrm{\scriptscriptstyle PK}}}
\def\CPKZ{C_{\mathrm{\scriptscriptstyle PK0}}}
\def\CP{C_{\mathrm{\scriptscriptstyle P}}}
\def\CSP{C_{\mathrm{\scriptscriptstyle SP}}}

\def\CPL{C_{\mathrm{\scriptscriptstyle PL}}}
\def\CLPL{C_{\mathrm{\scriptscriptstyle LPL}}}
\def\CRPL{C_{\mathrm{\scriptscriptstyle RPL}}}
\def\CPKL{C_{\mathrm{\scriptscriptstyle PKL}}}

\def\CB{C_{\mathrm{\scriptscriptstyle B}}}
\def\Cphi{C_\phi}
\newcommand{\norm}[1]{\left\|#1\right\|}

\def\sep#1{\left(#1\right)}
\def\seq#1{\left<#1\right>}
\newcommand{\be}{\begin{equation}}
\newcommand{\ee}{\end{equation}}
\newcommand{\ba}{\begin{aligned}}
\newcommand{\ea}{\end{aligned}}
\newcommand{\beqn}{\begin{equation}}
\newcommand{\eeqn}{\end{equation}}
\newcommand{\bear}{\begin{eqnarray}}
\newcommand{\eear}{\end{eqnarray}}
\newcommand{\bean}{\begin{eqnarray*}}
\newcommand{\eean}{\end{eqnarray*}}

\def\eps{{\varepsilon}}

\def\ie{{\em i.e.}}
\def\eg{{\em e.g.}}

\newcommand{\VA}{\mathfrak{V}_\antiSigma}

\renewcommand{\(}{\left(}
\renewcommand{\)}{\right)}
\renewcommand\P{{\mathbb P}}
\newcommand{\Id}{\mathrm{Id}}
\def\Spec{\hbox{\rm Sp}}

\DeclareMathOperator{\Divphi}{div_{\phi}}
\def\ccc{{\mathcal C}}
\def\ddd{{\mathcal D}}
\def\fM{\mathfrak M}
\def\fP{\mathfrak{P}}
\newcommand{\mmm}{\mathcal M}
\def\abs#1{\left\vert#1\right\vert}
\newcommand{\nablaphi}{\nabla_{\!\phi}}
\newcommand*\circled[1]{\tikz[baseline=(char.base)]{\node[shape=circle,draw,inner sep=0.8pt] (char) {#1};}}

\definecolor{darkgreen}{rgb}{0,0.4,0}

\title[Weighted Korn and Poincar\'e-Korn inequalities and
associated operators]{Weighted Korn and Poincar\'e-Korn
  inequalities in the Euclidean space and associated operators}

\author[K.~Carrapatoso]{Kleber Carrapatoso}

\address[K.~Carrapatoso]{CMLS (CNRS UMR n$^\circ$ 7640), \'Ecole
  Polytechnique, Institut Polytechnique de Paris, 91128 Palaiseau
  Cedex, France}

\email{kleber.carrapatoso@polytechnique.edu}

\author[J.~Dolbeault]{Jean Dolbeault}

\address[J.~Dolbeault]{CEREMADE (CNRS UMR n$^\circ$ 7534), PSL
  university, Universit\'e Paris-Dauphine, Place de Lattre de
  Tassigny, 75775 Paris 16, France}

\email{dolbeaul@ceremade.dauphine.fr}

\author[F. H\'erau]{Fr\'ed\'eric H\'erau}

\address[F. H\'erau]{LMJL (CNRS UMR n$^\circ$ 6629), 2, rue de la
  Houssini\`ere, Universit\'e de Nantes BP 92208, F-44322 Nantes
  Cedex 3, France}

\email{frederic.herau@univ-nantes.fr}

\author[S.~Mischler]{St\'ephane Mischler}

\address[S.~Mischler]{CEREMADE (CNRS UMR n$^\circ$ 7534), PSL
  university, Universit\'e Paris-Dauphine, Place de Lattre de
  Tassigny, 75775 Paris 16, France}

\email{mischler@ceremade.dauphine.fr}

\author[C.~Mouhot]{Cl\'ement Mouhot}

\address[C.~Mouhot]{DPMMS,
  Center for Mathematical Sciences, University of Cambridge,
  Wilberforce Road, Cambridge CB3 0WA, UK}

\email{C.Mouhot@dpmms.cam.ac.uk}

\date{\today}

\subjclass[2020]{Primary:
  \href{https://mathscinet.ams.org/mathscinet/search/mscbrowse.html?sk=default&sk=49J40&submit=Chercher}{49J40}.
  Secondary:
  \href{https://mathscinet.ams.org/mathscinet/search/mscbrowse.html?sk=default&sk=46E35&submit=Chercher}{46E35},
  \href{https://mathscinet.ams.org/mathscinet/search/mscbrowse.html?sk=default&sk=49Q20&submit=Chercher}{49Q20}}


\begin{document}

\begin{abstract}
  We prove functional inequalities on vector fields
  $u : \R^d \to \R^d$ when $\R^d$ is equipped with a bounded
  measure $e^{-\phi} \dd x$ that satisfies a Poincar\'e
  inequality, and study associated self-adjoint operators. The
  {\em weighted Korn inequality} compares the differential matrix
  $D u$, once projected orthogonally to certain
  finite-dimensional spaces, with its symmetric part $D^s u$ and,
  in an improved form of the inequality, an additional term
  $\nabla\phi\cdot u$. We also consider {\em Poincar\'e-Korn
    inequalities} for estimating a projection of $u$ by $D^s u$
  and zeroth-order versions of these inequalities obtained using
  the Witten-Laplace operator. The constants depend on geometric
  properties of the potential $\phi$ and the estimates are
  quantitative and constructive. These inequalities are motivated
  by kinetic theory and related with the {\em Korn inequality}
  (1906) in mechanics, which compares $D u$ and $D^s u$ on a
  bounded domain.
  % [Text only] We prove functional inequalities on vector fields
  % on the Euclidean space when it is equipped with a bounded
  % measure that satisfies a Poincaré inequality, and study
  % associated self-adjoint operators. The weighted Korn
  % inequality compares the differential matrix, once projected
  % orthogonally to certain finite-dimensional spaces, with its
  % symmetric part and, in an improved form of the inequality, an
  % additional term. We also consider Poincaré-Korn inequalities
  % for estimating a projection of the vector field by the
  % symmetric part of the differential matrix and zeroth-order
  % versions of these inequalities obtained using the
  % Witten-Laplace operator. The constants depend on geometric
  % properties of the potential and the estimates are
  % quantitative and constructive. These inequalities are
  % motivated by kinetic theory and related with the Korn
  % inequality (1906) in mechanics, on a bounded domain.
\end{abstract}

\maketitle
\thispagestyle{empty}
\vspace*{-0.5cm}

%%%%%%%%%%%%%%%%%%%%%%%%%%%%%%%%%%%%%%%%%%%%%%%%%%%%%%%%%%%%%%%%
%%%%%%%%%%%%%%%%%%%%%%%%%%%%%%%%%%%%%%%%%%%%%%%%%%%%%%%%%%%%%%%%
\section{Introduction and main results}
\label{Sec:intro}

%%%%%%%%%%%%%%%%%%%%%%%%%%%%%%%%%%%%%%%%%%%%%%%%%%%%%%%%%%%%%%%%
\subsection{The problem at hand}

Korn's inequality~\cite{Kor06,Kor08,Kor09} is a classical tool in
continuum mechanics which asserts the control of the $L^2$ norm
of the gradient of a vector field defined on a smooth bounded
domain~$\Omega$ of $\R^d$ by the $L^2$ norm of its symmetric
part:
\begin{equation}\label{eq:KornOmega}
  \|\nablaD u\|_{L^2(\Omega)}^2 \leq 2\,\|\nablasym
  u\|_{L^2(\Omega)}^2,\quad \forall\,u \in C^2(\overline\Omega;
  \R^d) \ \text{ such that } \ u=0 \text{ on } \partial\Omega\,.
\end{equation}
If $\mathfrak M$ (resp.~$\mathfrak M^s$) is the set of
$d \times d$ real (resp.~symmetric) matrices, $D u$
(resp.~$D^s u$) is the differential of $u$ (resp.~its symmetric
part) and takes values in $\mathfrak M$ (resp.~$\mathfrak
M^s$). Written with cartesian coordinates, this means
\begin{equation*}
  (D^su)_{ij}=\frac12\( \partial_j u_i + \partial_i u_j\)
  \quad\mbox{and}\quad
  (D^au)_{ij}=\frac12\( \partial_j u_i - \partial_i u_j\).
\end{equation*}
We denote by $A^a \in \mathfrak M^a$ the skew-symmetric part of
$A \in \mathfrak M$ and by $A^s=A-A^a$ its symmetric part. The
original proof of~(\ref{eq:KornOmega}) in~\cite{Kor06} is
instructive: it is enough to integrate over $\Omega$ the
pointwise identities
\begin{equation}\label{eq:Pointwise1}
  |\nablaskew u|^2-|\nablasym u|^2+(\nabla\cdot u)^2-\nabla \cdot
  \Big[ u\,(\nabla\cdot u)-(u\cdot \nabla )\,u \Big]=0 \quad
  \text{ and }\quad |\nablaD u|^2=|\nablasym u|^2+|\nablaskew u|^2
\end{equation}
and use the boundary condition to get
$2\,\|\nablasym u\|_{L^2(\Omega)}^2=\|\nablaD
u\|_{L^2(\Omega)}^2+\|\nabla\cdot u\|_{L^2(\Omega)}^2$, where
$\nabla\cdot u$ denotes the divergence of $u$ and $L^2(\Omega)$
is the $L^2$ norm for matrix-valued, vector-valued or real-valued
functions.

The boundary condition $u=0$ is a severe restriction. In view of
applications in kinetic theory, Desvillettes and Villani
in~\cite{DV02} enlarge the set of possible vector fields to those
satisfying only $u\cdot \bn=0$ on the boundary, where $\bn$
denotes the outward normal unit vector to $\D\Omega$. The set of
{\em infinitesimal rotations}
$$
\RR:=\set{ R : x \in \R^d \mapsto \antiSigma\,x \in \R^d \ \text{with} \ \antiSigma\in\mathfrak M^a}
$$
is a family of vector fields which plays a key role. If $\Omega$
has some rotational invariance, then some infinitesimal rotations
satisfy the boundary condition while their symmetric differential
is zero. Being invariant under the action of a group of rotations
$t \to e^{tA}$ for a given $A \in \mathfrak M^a$ means that
\begin{equation*}
\forall\,t \in \R\,,\quad e^{tA}\,\Omega=\Omega
\end{equation*}
(here we suppose that $\Omega$ is invariant under rotations
centred at $0$ without loss of generality). Taking the derivative
with respect to $t$ shows that the set of {\em infinitesimal
  rotations preserving $\Omega$} is
\begin{equation*}
  \RR_\Omega :=\big\{ R\in\RR\,:\,\forall\,x\in\D \Omega\,,\;
  \bn(x)\cdot R(x)=0 \big\},
\end{equation*}
where we implicitly use the fact that skew-symmetric matrices
generate the tangent space of the orthogonal
group. In~\cite[Inequality~(38)]{DV02}, Desvillettes and Villani
state the following {\em Korn inequality}
\begin{equation}\label{eq:KornOmega2}
\inf_{R\in\RR_\Omega}\|\nablaD (u-R)\|^2_{L^2(\Omega)} \le C_\Omega\,\|D^s u\|^2_{L^2(\Omega)}\,,\quad \forall\,u \in C^2(\overline\Omega; \R^d) \ \text{ such that } \ u\cdot \bn=0 \text{ on } \partial\Omega\,,
\end{equation}
which takes into account invariances by rotation. They obtain
quantitative estimates on the constant
$C_\Omega$. Inequality~\eqref{eq:KornOmega2} is an important
ingredient in~\cite{Desvillettes-Villani-2005} to prove
hypocoercivity for the Boltzmann equation in $\Omega$.

\medskip In this article, our aim is to establish similar {\em
  Korn and related inequalities}, with constructive constants, in
the whole Euclidean space~$\R^d$ in presence of a {\em confining
  potential} $\phi : \R^d \to \R$, \ie, in the $L^2$ space with
reference measure $e^{-\phi(x)} \dd x$. Our motivation comes from
the hypocoercivity theory of kinetic operators with more than one
microscopic invariant studied in~\cite{CDHMMShypo}, but the
inequalities are of independent interest.

%%%%%%%%%%%%%%%%%%%%%%%%%%%%%%%%%%%%%%%%%%%%%%%%%%%%%%%%%%%%%%%%%%%%%%
\subsection{Assumptions and notations}\label{subsec:notation}

We consider a potential \hbox{$\phi : \R^d \to \R$}, $d \ge 2$
satisfying the conditions:
\\[4pt]
\circled1 the measure $e^{-\phi(x)} \dd x$ is a centred
probability measure
\begin{equation}
  \label{hyp:intnorm} \tag{H1}
  \int_{\R^d}e^{-\phi(x)} \dd x=1\quad \mbox{and}\quad \int_{\R^d}
  x\,e^{-\phi(x)} \dd x=0\,,
\end{equation}
\\[4pt]
\circled2 the potential $\phi$ is of class $C^2(\R^d;\R)$ and,
for all $\eps>0$, there exist a constant $C_\eps$ such that
\begin{equation}
  \label{hyp:regularity} \tag{H2}
  \begin{array}{c}
    \forall\,x\in \R^d,\quad |D^2\phi(x)| \leq
    \eps\,|\nabla\phi(x)|^2+C_\eps\,,
  \end{array}
\end{equation}
\\[4pt]
\circled3 the measure $e^{-\phi} \dd x$ satisfies the Poincar\'e
inequality with constant $C_P$: for all scalar functions~$f$ in
the space $\ccc_c^\infty(\R^d;\R)$ of smooth functions with
compact support, we have
\begin{equation}
  \label{hyp:poincarebasique} \tag{H3}
  \int_{\R^d}\left|f(x)-\seq f\right|^2\,e^{-\phi(x)} \dd x\leq \CP
  \int_{\R^d} |\nabla f(x)|^2\,e^{-\phi(x)} \dd
  x\,,\quad\mbox{where}\quad\seq f:=\int_{\R^d}
  f(x)\,e^{-\phi(x)}\dd x\,.
\end{equation}
Assumption~\eqref{hyp:intnorm} is a classical integrability
condition on $\phi$. The fact that the center of mass
$\int_{\R^d} x\,e^{-\phi(x)} \dd x$ is finite is a consequence
of~\eqref{hyp:poincarebasique} applied with $f(x)=x_i$,
$i=1,\ldots,d$. There is no loss of generality in choosing
$\int_{\R^d} x\,e^{-\phi(x)} \dd
x=0$. Assumption~\eqref{hyp:regularity} is a regularity
assumption at infinity which, in the language of operator theory
and in a suitable functional framework, says that the
multiplication operator by $|D^2 \phi(x)|$ is {\em
  infinitesimally bounded} by the multiplication operator by
$|\nabla\phi(x)|^2$ (see~\cite[Chapter~X]{RS75}). Here no growth
assumption is made on
$|\nabla\phi|$. Assumption~\eqref{hyp:regularity} is satisfied
for instance if
\begin{equation*}
  \sup_{x \in \R^d} \frac{ D^2
    \phi(x)}{\sqrt{1+|\nabla\phi(x)|^2}} < \infty\quad
  \textrm{or}\quad \lim_{|x| \rightarrow \infty} \frac{ D^2
    \phi(x)}{1+|\nabla\phi(x)|^2}=0\,.
\end{equation*}
Assumption~\eqref{hyp:poincarebasique} can be interpreted as a
{\em measure concentration} property: it implies that
\begin{equation}
  \label{Concentration}
  \int_{\R^d} |\nabla\phi(x)
  |^2\,e^{-\phi(x)} \dd
  x<\infty\quad\mbox{and}\quad\forall\,k\in\N\,,\quad\int_{\R^d}
  |x|^{2k}\,e^{-\phi(x)} \dd x<\infty
\end{equation}
by~\eqref{hyp:regularity} for the first estimate and by an easy
induction (see Remark~\ref{rem:concentration} in
Appendix~\ref{Sec:Remarks}) for the
second~one. Assumptions~\eqref{hyp:intnorm},~\eqref{hyp:regularity}
and~\eqref{hyp:poincarebasique} are satisfied, for instance, if
$\phi\in C^2(\R^d;\R)$ and either
$\phi(x)=\alpha\,|x|^\gamma+\beta$ with $\gamma \ge 1$ or
$\phi(x)=\alpha\,e^{|x|^2}+\beta$, for large values of $|x|$,
where $\alpha>0$ and $\beta$ are two parameters. The three
assumptions are also satisfied by the {\em normalized Gaussian}
defined~by \be\label{Gaussian} \forall\,x \in \R^d,\quad
\phi(x)=\frac12\,|x|^2+\frac d2\,\ln(2\pi)\,.  \ee Associated
with $\phi$ and thanks to~\eqref{hyp:regularity}, there exists
two constants $\Cphi>8$ and $\Cphi'>\Cphi$ such that
\begin{equation}
  \label{eq:cphi}
  \forall\,x \in \R^d,\quad 4\,\sqrt d\,|D^2 \phi(x)|
  \leq |\nabla\phi(x)|^2+\Cphi -1
  \quad \mbox{and}\quad
  4\,|D^2 \phi(x)| \leq
  \Cphi^{-1/2}\,\big(|\nabla\phi(x)|^2+\Cphi'\big)\,.
\end{equation}

\medskip As in the case of a bounded domain, the set $\RR$ of
{\em infinitesimal rotations} plays a key role in the study of
Korn inequalities in the whole space. The symmetric differential
applied to an infinitesimal rotation is zero. In our setting, the
invariance under the action of a group of rotations
$t \to e^{tA}$ for a given $A \in \mathfrak M^a$ means
\begin{equation*}
  \label{eq:defrotinv}
  \forall\,t \in \R\,,\quad \forall\,x \in \R^d,\quad
  \phi\(e^{tA}x\)=\phi(x)\,,
\end{equation*}
where, again, we implicitly use the assumption that the measure
is centred. Differentiating the above identity with respect to
$t$ yields that the set of {\em infinitesimal rotations
  preserving $\phi$} is
\begin{equation*}
  \RR_\phi :=\big\{ R\in\RR\,:\,\forall\,x\in\R^d,\;\nabla\phi(x)
  \cdot R(x)=0 \big\}\,.
\end{equation*}
This set is a central geometric objet in our analysis. In the
inequalities, the addition of a term involving
$\nabla\phi\cdot R$ allows us to control the infinitesimal
rotations for which $\phi$ is {\em not} invariant, as we shall se
later.

\medskip In this article, we adopt the following conventions. We
denote by $|\cdot |$ the Euclidean norm in $\R$, $\R^d$
and~$\mathfrak M$, by $a\cdot b$ the scalar product of two
vectors in $\R^d$ and by $A : B$ the scalar product of two
matrices $A$ and~$B$ seen as vectors in $\R^{d^2}$. We denote by
$\|\cdot\|$ the $L^2$ norm corresponding to~$|\cdot |$ and weight
$e^{-\phi} \dd x$, and by $(\cdot,\cdot)$ the corresponding
scalar product, that is,
\begin{equation*}
  (f,f)=\|f\|^2=\int_{\R^d}|f(x)|^2\,e^{-\phi} \dd x
\end{equation*}
and we will refer to $L^2$ indifferently for functions with
values in $\R$, $\R^d$ and~$\mathfrak M$. We shall use
$\langle\cdot \rangle$ for the average (component by component)
according to the measure $e^{-\phi} \dd x$ of functions with
values in $\R$, $\R^d$ and~$\mathfrak M$. We use the notation
$\nabla$ for the gradient of scalar functions (with values in
$\R^d$) and $D$ for the gradient of vector fields (with values in
$\mathfrak M$). We denote by $H^1$ the space of functions $f$ or
(when there is no ambiguity) vector fields $u$ such that
respectively $f$ and $\nabla f$ or $u$ and $Du$ are in $L^2$. The
space $H^{-1}$ is the dual of $H^1$ with respect to the $L^2$
scalar product. The weight function $\wgt$ is defined by
$$
\wgt:=\sqrt{1+|\nabla\phi|^2}\,.
$$

Let $\P$ be the orthogonal projection of vector-valued functions
in $L^2$ onto $\RR$, and $\P_\phi$ the orthogonal projection onto
$\RR_\phi$. We denote by $\RR_\phi^\perp$ the orthogonal vector
space to $\RR_\phi$ in $L^2$ and
$\RR_\phi^c=\RR \cap \RR_\phi^\bot$ the restriction of the
orthogonal space to $\RR_\phi$ in $\RR$ or, in other words,
$\RR_\phi^c=\P\,\RR_\phi^\bot$. For instance if $\phi$ has {\em
  no} invariance by any rotation $e^{tA}$ then $\RR_\phi=\{ 0\}$,
and if $\phi$ is radially symmetric then $\RR_\phi=\RR$. Let
$\fP$ be the orthogonal projection of matrix-valued functions in
$L^2$ onto the set of constant antisymmetric matrices
$\fM^a=D\RR$. For all vector field $u \in H^1$, we have
\begin{equation}
  \label{ExplProj}
  \fP(Du)=\seq{D^a u}.
\end{equation}
We also denote by $\fP_\phi$ the $L^2$ orthogonal projector onto
$\fM_\phi:=D \RR_\phi$ and by $\fM_\phi^c$ the orthogonal of
$\fM_\phi$ in~$\fM^a$, {\em i.e.},
$\fM_\phi^c=\fP\fM_\phi^\bot$. The projections are summarised in
Figure~\ref{fig:diagram}. Note that $D \RR_\phi^c$ and
$\fM_\phi^c$ generically differ since the inner products
underlying the two orthogonal decomposition are different.

\begin{figure}[h]
  \includegraphics[width=7cm]{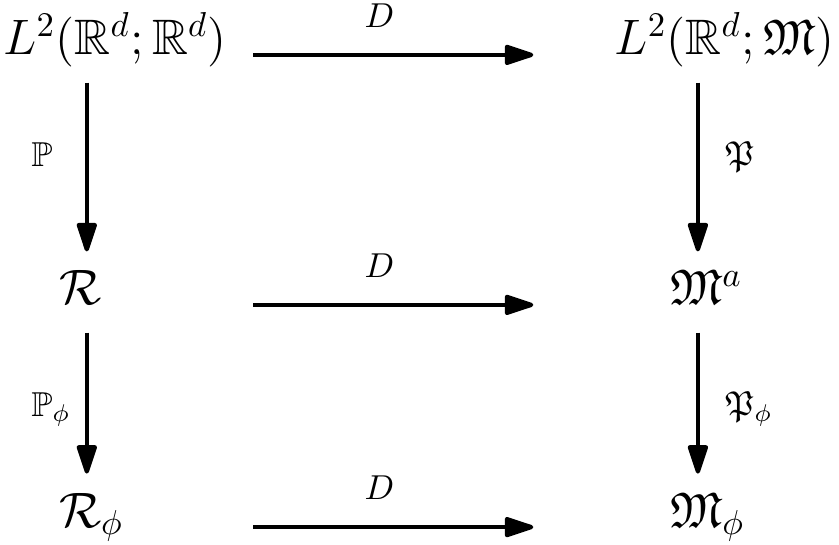}
  \label{fig:diagram}
  \caption{Representation of the orthogonal
    decompositions.}
\end{figure}

One additional notation will be used throughout this paper: if
$x$ and $y$ are two vectors in $\R^d$, we denote by $x \otimes y$
the matrix $(x_i\,y_j)_{1\le i,j\le d}$. Further details on
$D \RR_\phi^c$ and $\fM_\phi^c$ are collected in
Appendix~\ref{Appendix:B2}.

%%%%%%%%%%%%%%%%%%%%%%%%%%%%%%%%%%%%%%%%%%%%%%%%%%%%%%%%%%%%%%%%
\subsection{Main results}\label{Sec:Main}

All inequalities in this paper are quantitative with explicit
estimates on the constants. The first result is the counterpart
of~\eqref{eq:KornOmega2} in the whole Euclidean space, with some
additional consequences based on Poincar\'e inequalities. The
statement involves the whole set of {\em infinitesimal rotations}
$\RR$.
%---------------------------------------------------------------
\begin{theo}[Korn, Poincar\'e-Korn and strong Poincar\'e-Korn
  inequalities]
  \label{theo:KPK}
  Suppose~\eqref{hyp:intnorm}--\eqref{hyp:regularity}--\eqref{hyp:poincarebasique}. Then
  there are a Korn constant $\CK$, a Poincar\'e-Korn constant
  $\CPK $ and a strong Poincar\'e-Korn constant $\CSPK$ with
  explicit bounds involving only $\CP$, $\Cphi$ and $\Cphi'$ such
  that, for all $u\in H^1$,
  \begin{align}
    \label{eq:WKfull}
    & \inf_{R\in\RR}\|\nablaD (u-R)\|^2=\|\nablaD u-\fP (Du)\|^2
      \leq \CK\,\|D^s u\|^2,\\
    \label{eq:WPKfull}
    & \inf_{R\in\RR}\|u-\seq{u}-R\|^2=\|u-\seq{u}-\P (u)\|^2
      \leq \CPK\,\|D^s u\|^2,\\
    \label{eq:WPKstrong}
    & \big\|\wgt\,\big(u-\seq{u}-\P (u)\big)\big\|^2 \leq
      \CSPK\,\|D^s u\|^2.
  \end{align}
  Moreover in the Gaussian case~\eqref{Gaussian}, optimal
  constants are $\CK=4$, and $\CPK=2$.
\end{theo}
%---------------------------------------------------------------------
The terminology {\em ``Korn constant''} refers to Korn's original
results~\cite{Kor06,Kor08,Kor09} whereas {\em
  ``Poincar\'e-Korn''} and {\em ``strong Poincar\'e-Korn''}
respectively refer to usual Poincar\'e inequalities and to strong
Poincar\'e inequalities (see
Proposition~\ref{prop:poincares}). For brevity, we shall speak
generically of {\em ``Korn-type inequalities''} or simply {\em
  ``Korn inequalities''}. Explicit bounds for $\CK$, $\CPK$ and
$\CSPK$ will also be given later. The constant $\CP$ is the
optimal constant in~\eqref{hyp:poincarebasique} while $\Cphi$ and
$\Cphi'$ refer to~\eqref{eq:cphi}. The minimum in the left-hand
side of~\eqref{eq:WKfull} is explicit: according
to~\eqref{ExplProj}, we have
$$
\|\nablaD u-\fP (Du)\|^2=\|\nablaD u-\seq{D^a u}\|^2 \leq \CK\,\|D^s u\|^2.
$$
In~\eqref{eq:WPKfull} and~\eqref{eq:WPKstrong}, there is no
simple expression for $\P(u)$ as for $\fP(Du)$. Note that the
strong Poincar\'e inequality (see
Proposition~\ref{prop:poincares} below) implies $\RR\subset H^1$,
so that the statement of Theorem~\ref{theo:KPK} makes sense.

\medskip The defaults of axisymmetry of the boundary of the
domain were taken into account in~\cite{DV02}, in the bounded
domain case {\em without potential} (this case will be called the
{\em ``flat case''} from now on). Here the eventual
non-axisymmetry arises from the potential. Measuring the default
of axisymmetry motivates our introduction of the finite
dimensional space $\RR_\phi^c$ and of the {\em rigidity of vector
  fields} constant~$\CRV$ defined by
\begin{equation}
  \label{eq:rigidityvect}
  \CRV^{-1}
  :=\min_{ A\,x+b \: \in \: (\RR_\phi^c \oplus \R^d)\setminus \set{0}}
  \frac{\norm{\nabla\phi(x)\cdot (A\,x+b)}^2}{\norm{A\,x+b\,}^2}
\end{equation}
if $\RR_\phi^c\neq\set{0}$, and of the {\em rigidity of
  differential} constant $\CRD$ defined by
\begin{equation}
  \label{eq:rigiditydiff}
  \CRD^{-1} =\min_{ (A,b) \: \in \: (\fM_\phi^c\otimes\R^d)
    \setminus
    \set{(0,0)}}
  \frac{\norm{\nabla\phi(x)\cdot (A\,x+b) }^2}{|A|^2+|b|^2}
\end{equation}
if $\fM_\phi^c\neq\set{0}$. We adopt the convention that $\CRV=0$
(respectively $\CRD=0$) if $\RR_\phi^c=\set{0}$ (respectively
$\fM_\phi^c=\set{0}$). Let us show that these constants are
well-defined in $\R_+$. Since $\RR$ and $\fM_\phi$ have a finite
dimension the minima in~\eqref{eq:rigidityvect}
and~\eqref{eq:rigiditydiff} exist. The first (respectively
second) minimum is positive when $\RR_\phi^c \neq\set{0}$
(respectively $\fM_\phi^c \neq\set{0}$). Indeed the linear maps
$\RR_\phi^c\oplus \R^d : A\,x+b \mapsto \nabla\phi(x)\cdot
(A\,x+b) \in L^2$ and
$\fM_\phi^c\oplus \R^d : (A,b) \mapsto \nabla\phi(x)\cdot
(A\,x+b) \in L^2$ are injective: if
$\nabla\phi(x)\cdot (A\,x+b)=0$ for all $x \in \R^d$, then by
integration by parts
\begin{equation*}
  0=\int_{\R^d} \nabla\phi(x)\cdot (A\,x+b)\ b\cdot x\,
  e^{-\phi(x)} \dd x=|b|^2
\end{equation*}
because $\seq{x}=0$, so that $b=0$, and as a consequence $A=0$
since $\RR_\phi \cap \RR_\phi^c=\set{0}$ and
$\fM_\phi \cap \fM^c_\phi = \{0\}$. We can now state {\it
  precised} Korn and Poincar\'e-Korn inequalities in which $u$ is
also controlled on the space of infinitesimal rotation that do
not leave $\phi$ invariant.
%---------------------------------------------------------------
\begin{theo}[Precised Poincar\'e-Korn and Korn inequalities]\label{theo:PKPK} Suppose~\eqref{hyp:intnorm}--\eqref{hyp:regularity}--\eqref{hyp:poincarebasique}. Then there are a precised Korn constant $\CK'$ and a precised Poincar\'e-Korn constant $\CPK'$ such that, for all $u\in H^1$,
\begin{align}
\label{eq:WPK} & \inf_{R\in\RR_\phi}\|u-R\|^2=\|u-\P_\phi (u)\|^2 \leq \CPK'\,\|D^s u\|^2+2\,\CRV\,\|\nabla\phi\cdot u\|^2,\\
\label{eq:WK} & \inf_{R\in\RR_\phi}\|\nablaD (u-R)\|^2=\|\nablaD u-\fP_\phi (Du)\|^2 \le \CK'\,\|D^s u\|^2+2\,\CRD\,\|\nabla\phi\cdot u\|^2.
\end{align}
Moreover the constants $\CPK'$ and $\CK'$ have explicit bounds depending only on the structural constants $\CP$, $\Cphi$, $\Cphi'$, $\CRD$ and $\CRV$. \end{theo}
%---------------------------------------------------------------------
Explicit bounds for $\CK'$ and $\CPK'$ will be given in the proofs. For any $u\in H^1$, the strong Poincar\'e inequality implies $\nabla\phi\cdot u \in L^2$ thanks to~\eqref{Concentration}, so that the statement makes sense (see Proposition~\ref{prop:poincares} below and Remarks~\ref{rem:concentration} and~\ref{rem:integrability} in Appendix~\ref{Sec:Remarks}). The main difference with the flat case is the $\norm{\nabla\phi\cdot u}^2$ term in the right-hand side of the inequality whereas there is no additional term in~\eqref{eq:KornOmega2}. This cannot be avoided as shown by the following example. Let us consider an infinitesimal rotation $u=R$, with $R\neq 0$, in the case without invariance by any rotation $e^{tA}$, that is, $\RR_\phi=\set{0}$. Then $u \in H^1$ and inequality~\eqref{eq:WK} reduces to $0 \neq \norm{R}^2 \leq 2\,\CRV\,\norm{\nabla\phi\cdot R}^2$ since $\nablasym R=0$. This also shows that, compared to~\eqref{eq:WKfull}, an additional term is needed.

\medskip As often, the functional inequalities of Theorems~\ref{theo:KPK} and~\ref{theo:PKPK} are linked with spectral properties of nonnegative differential operators. By a simple integration by parts, the formal adjoint of $\nabla$ in $L^2$ equipped with the weight $e^{-\phi(x)} \dd x$ is $\nabla^* u=-\,\nablaphi\cdot u$ for any smooth vector field $u$, where $\nablaphi\cdot u :=\nabla\cdot u- \nabla\phi\cdot u$. The first operator is the so-called {\em Witten-Laplace operator on functions} $-\Delta_\phi$ (sometimes also called the {\em Ornstein-Uhlenbeck operator}) which replaces the usual Laplacian in the flat case. It is associated with the quadratic form $f\mapsto\|\nabla f\|^2=\int_{\R^d} |\nabla f|^2\,e^{-\phi} \dd x$ and defined by
\be\label{def:Deltaphi}
-\Delta_\phi f :=-\,\nablaphi\cdot \nabla f=-\,\Delta f+\nabla\phi\cdot \nabla f\,.
\ee
The operator $-\Delta_\phi$ is nonnegative and symmetric.

For convenience, we shall also denote by $-\Delta_\phi$ the
operator acting coordinate by coordinate on vector fields, that
is for any smooth vector field $u$,
$(\Delta_\phi u)_i=\Delta_\phi u_i$, and similarly extend it to
matrices. In the same spirit, we introduce various differential
operators. The formal adjoint of $D^s$ is defined by
$(D^s)^*\mathfrak F=-\,D^s_\phi\cdot\mathfrak F$ for any
matrix-valued function $\mathfrak F$, so that
$D^s_\phi:=D^s-D^s\phi$ acts on matrix-valued functions and takes
value in a space of vector fields. Here
$D^s\phi\cdot\mathfrak F:=\nabla\phi\cdot\mathfrak F^s$. Let us
consider
\begin{equation*}\label{eq:SWS}
  -\Delta_S\,u:=-\,D^s_\phi\cdot D^su\quad\mbox{and}\quad
  -\Delta_{S\phi}\,u:=-\,D^s_\phi\cdot D^s u+(\nabla\phi
  \otimes \nabla\phi)\,u\,,
\end{equation*}
acting on a smooth vector field $u$. The differential operators
$-\Delta_S$ and $-\Delta_{S\phi}$ are associated respectively
with $\|D^s u\|^2$ and $\|D^s u\|^2+\|\nabla\phi\cdot u\|^2$,
which appear in the various Korn and Poincar\'e-Korn
inequalities. Additional details have been collected in
Appendix~\ref{Appendix:B4}.
%--------------------------------------------------------------
\begin{theo}[Associated operators acting on vector fields]
  \label{theo:ao}
  Suppose~\eqref{hyp:intnorm}--\eqref{hyp:regularity}--\eqref{hyp:poincarebasique}. Then
  the operators $-\Delta_\phi$, $-\Delta_S$, and
  $-\Delta_{S\phi}$ are essentially self-adjoint on $L^2$. They
  have a common domain $\ddd$, finite dimensional kernels
  \begin{equation*}\label{kernel}
    \ker(-\Delta_\phi)=\R^d,\quad
    \ker(-\Delta_S)=\R^d\oplus\RR\,,
    \quad\ker(-\Delta_{S\phi})=\RR_\phi\,,
  \end{equation*}
  and positive spectral gaps. The spectral gap of $-\Delta_\phi$
  is the Poincar\'e constant $\CP$ while the spectral gaps of
  $-\Delta_S$ and $-\Delta_{S\phi}$ are estimated respectively in
  Theorems~\ref{theo:KPK}~and~\ref{theo:PKPK}.
\end{theo}
%--------------------------------------------------------------
A positive spectral gap means that the infimum of the restriction
of the spectrum to $(0,+\infty)$ is positive. Our last main
result is devoted to a Korn-type inequality valid for vector
fields $u\in L^2$ while Theorems~\ref{theo:KPK}
and~\ref{theo:PKPK} are limited to $u\in H^1$. We shall
compose by inverse powers of the following positive operator
\begin{equation}
  \label{Lambda}
  \Lambda:=-\,\Delta_\phi+\Id
\end{equation}
acting on functions, vector fields or matrices, {\em coordinate
  by coordinate}. By Theorem~\ref{theo:ao}, $\Lambda$ is
essentially self-adjoint (we keep the same name for the unique
self-adjoint extension), $\Lambda\ge\Id$, and $\Lambda^{-1/2}$ is
one-to-one from $H^{-1}$ into~$L^2$ (see
Propositions~\ref{prop:H1} and~\ref{prop:domain} for more
details). In order to measure the possible non-axisymmetry of the
potential $\phi$ in an $L^2$ setting, we introduce the {\em
  rigidity of vector fields} constant 
\begin{equation}\label{eq:rigidityzero}
\CRVZ^{-1}:=\min_{ A\,x+b \: \in \: (\RR_\phi^c \oplus
  \R^d)\setminus \set{0}} \frac{\norm{\Lambda^{-1/2}\,
    \nabla\phi(x)\cdot (A\,x+b)}^2}{\norm{A\,x+b}^2}\,.
\end{equation}
when $\RR_\phi^c\neq\set{0}$ and, by convention, $\CRVZ:=0$ if
$\RR_\phi^c=\set{0}$. The proof that this constant $\CRVZ$ is
well-defined in $\R_+$ is exactly similar to that for $\CRV$.
%---------------------------------------------------------------
\begin{theo}[Zeroth order Korn and Poincar\'e-Korn
  inequalities]\label{theo:KPK0}
  Suppose~\eqref{hyp:intnorm}--\eqref{hyp:regularity}--\eqref{hyp:poincarebasique}. Then
  there are a zeroth order Korn constant $\CKZ$ and a zeroth
  order Poincar\'e-Korn constant $\CPKZ$ with explicitly
  computable bounds depending only on $\phi$ such that, for all
  $u\in L^2$,
  \begin{align}
    \label{eq:WKZfull}
    & \inf_{R\in\RR}\|\Lambda^{-1/2}\,\nablaD (u-R)\,\|^2=
      \big\|\Lambda^{-1/2}\,\big(\nablaD u-\fP (Du)\big)\big\|^2
      \leq
      \CKZ\,\|\Lambda^{-1/2}\,D^s u\|^2,\\
    \label{eq:WPKZfull}
    & \inf_{R\in\RR}\|u-\seq{u}-R\,\|^2=
      \norm{u-\seq{u}- \P(u)}^2 \le
      \CPKZ\,\|\Lambda^{-1/2}\,D^s u\,\|^2.
  \end{align}
  As a consequence, there is a zeroth order precised
  Poincar\'e-Korn constant $\CPKZ'$ such that, for all
  $u \in L^2$,
  \begin{align}
    \label{eq:WPKzero}
    & \inf_{R\in\RR_\phi}\|u-R\,\|^2=\|u-\P_\phi (u)\|^2
      \le \CPKZ'\,\|\,\Lambda^{-1/2}\,D^s u\,\|^2
      +2\,\CRVZ\,\|\Lambda^{-1/2}\,\nabla\phi\cdot u\,\|^2.
  \end{align}
\end{theo}
%---------------------------------------------------------------
Inequality~\eqref{eq:WPKzero} is a straightforward consequence of
the Poincar\'e-Korn inequality~\eqref{eq:WPKZfull} and the
existence of the rigidity constant $\CRVZ$.

%%%%%%%%%%%%%%%%%%%%%%%%%%%%%%%%%%%%%%%%%%%%%%%%%%%%%%%%%%%%%%%%%%%%%%
\subsection{Main tools and considerations on the optimal cases
  and optimal constants}

The paper relies on three main tools.

\smallskip\noindent\circled1 {\em Poincar\'e-Wirtinger and
  Poincar\'e-Lions inequalities.} The proof of
Theorems~\ref{theo:KPK} and~\ref{theo:PKPK} for vector fields
relies on Poincar\'e-Wirtinger inequalities for {\em scalar
  functions}, which go as follows.
%---------------------------------------------------------------------
\begin{prop}
  \label{prop:poincares}
  Assume that~\eqref{hyp:intnorm},~\eqref{hyp:regularity}
  and~\eqref{hyp:poincarebasique} hold, for some Poincar\'e
  constant $\CP$. Then there exists a {\em strong Poincar\'e}
  constant $\CSP>0$ such that
  \begin{equation}
    \label{eq:strongpoincare}
    \forall\,f\in H^1,\quad\|\wgt(f-\seq{f})\|^2 \le
    \CSP\,\|\nabla f\|^2
  \end{equation}
  with $\CSP \leq \Cphi\,(1+\CP)$. With $\Lambda$ as
  in~\eqref{Lambda}, there exists also a {\em Poincar\'e-Lions}
  constant $\CPL>0$ such that
  \begin{equation}
    \label{eq:poincarelions}
    \forall\,f\in L^2,\quad\norm{f-\seq{f}}^2 \leq \CPL\,
    \|\Lambda^{-1/2}\,\nabla f\|^2 \leq \CPL\,\norm{f-\seq{f}}^2.
  \end{equation}
\end{prop}
%--------------------------------------------------------------
Under the sole
assumptions~\eqref{hyp:intnorm}--\eqref{hyp:regularity}--\eqref{hyp:poincarebasique},
inequalities~\eqref{eq:strongpoincare}
and~\eqref{eq:poincarelions} are not completely standard. These
inequalities are linked to the spectral properties of
$-\Delta_\phi$, studied in Section~\ref{sec:wittenpoincare},
where elements of proofs of~\eqref{eq:strongpoincare}
and~\eqref{eq:poincarelions} are also collected. An estimate of
$\CPL$ is given in~\eqref{CPL}.

\smallskip\noindent\circled2 The {\em Schwarz Theorem} allows us
to write all components of the second-order differential of a
vector field~$u$ thanks to its symmetric components using the
identity
\begin{equation}\label{Schwarz}
  \forall\,i,j,k \in \set{1,\cdots, d}, \quad \partial_k
  \left( D^a u \right)_{ij}=\partial_j \left( D^s u \right)_{ik}
  -\partial_i \left( D^s u \right)_{jk}.
\end{equation}
This algebraic property is at the core of all Korn-type
inequalities, it means that derivatives of $D^au$ are in the span
of the derivatives of $D^s u$. Note that the Schwarz Theorem also
implies $D^a\,\nabla=0$ which is central in the construction of
the De Rham complex.

\smallskip\noindent\circled3 The {\em rigidity constants}, as
defined in~\eqref{eq:rigidityvect},~\eqref{eq:rigiditydiff}
and~\eqref{eq:rigidityzero}, measure the defects of axisymmetry
of the potential $\phi$. See Appendix~\ref{subsec:rigidity} for a
discussion.

\medskip Our method of proof can be summarised as follows: (i) we
take care of the finite-dimensional parts $\RR_\phi^c$ and
$\fM_\phi^c$ thanks to the rigidity constants in \circled3, (ii)
we apply twice the Poincar\'e inequality in \circled1, first in the
form~\eqref{eq:strongpoincare} and second in the
form~\eqref{eq:poincarelions}, so that we access second-order
derivatives but remain at first order thanks to $\Lambda^{-1/2}$,
(iii) we use the algebraic property in \circled2 to get rid of
the derivatives of $D^a u$.

\medskip The infima
in~\eqref{eq:WKfull},~\eqref{eq:WPKfull},~\eqref{eq:WPK},~\eqref{eq:WK},~\eqref{eq:WPKZfull}
and~\eqref{eq:WPKzero} are achieved respectively at
$\fP(Du)=\seq{D^a u}$ (see Section~\ref{Sec:Prf1}),
$\seq u+\P(u)$, $\P_\phi(u)$, $\fP_\phi(Du)$, $\seq u+\P(u)$ and
$\P_\phi(u)$ as a consequence of the definitions of the various
orthogonal projections. In the Gaussian case~\eqref{Gaussian}, we
have $D \P(u)=\fP (Du)$, but this relation is not true otherwise.
The constants are estimated explicitely and a summary is provided
in Appendix~\ref{Sec:Constants}.

%%%%%%%%%%%%%%%%%%%%%%%%%%%%%%%%%%%%%%%%%%%%%%%%%%%%%%%%%%%%%%%%%%%%%%
\subsection{A brief review of the literature and a conjecture}
\label{sec:review}

We refer to~\cite[Eq.~(13)]{MR0022750},~\cite[Chapter~3,
Section~3.3]{MR0521262},~\cite[page~291]{MR936420},
and~\cite{MR1368384} for statements of the original Korn
inequality which goes back to~\cite{Kor06,Kor08,Kor09} in a
bounded domain with Dirichlet conditions, and
to~\cite{MR3498171,MR3582590,MR3570353} for considerations on the
best constant. There is a huge literature on applications to the
Navier-Stokes equations and elasticity models, which is out of
the scope of the present paper: see~\cite{MR3294348} for an
introduction to Korn's inequality applied to these topics.

The case of Korn inequalities in bounded domains with {\em
  Neumann boundary conditions} was carried out in~\cite{DV02},
driven by applications in kinetic theory
in~\cite{Desvillettes-Villani-2005}. The proof relates the Korn
constant to the so-called {\em Grad number}, which is further
studied in~\cite{MR2542573} and related to other geometric
bounds. The notion of Grad's number goes back to~\cite{Grad65} in
a bounded domain and was used in~\cite{DV02}. We refer to
Appendix~\ref{subsec:rigidity} for a more detailed discussion and
how it relates to our rigidity constants.

In bounded domains, inequalities of type~\eqref{eq:WPKfull} are
usually called {\em Poincar\'e-Korn} estimates (see for
instance~\cite[Section~1.3.1]{MR3294348}), and
inequality~\eqref{eq:WK} is reminiscent of what is sometimes
called the {\em second Korn inequality:}
see~\cite[Inequality~(7)]{MR631678}
and~\cite[Theorem~2]{MR995908}.

To our knowledge, the only result in the whole space with a
confinement potential is~\cite[Section~5]{Duan_2011} where the
Korn inequality~\eqref{eq:WPKfull} is proved by compactness,
under an additional growth condition on $\nabla\phi$. The
original contributions of this paper are\\ \hspace*{12pt} (i) a
proof of weighted Poincar\'e-Korn and Korn inequalities, under
rather general conditions,\\ \hspace*{12pt} (ii) a constructive
method which provides us with quantitative estimates on the
constants,\\ \hspace*{12pt}
(iii) some optimal constants in the Gaussian case.\\
Our method is likely adaptable to bounded domains and also to
fractional inequalities in the spirit
of~\cite{MR2746437}. Inspired by the properties of the Gaussian
Poincar\'e inequality, {\em e.g.}, in~\cite{Courtade_2020}, we
finally make the following conjecture.
%---------------------------------------------------------------------
\begin{conj}[Optimal constants]
  For a given $\phi$
  satisfying~\eqref{hyp:intnorm}--\eqref{hyp:regularity}--\eqref{hyp:poincarebasique}
  with $\int_{\R^d} x_i\,x_j\,e^{-\phi(x)} \dd x=\delta_{ij}$ for
  all $i,j \in \set{1,\cdots, d}$ and
  $D^2 \phi \ge \text{{\normalfont Id}}$, one has $\CPK \ge 2$ and
  $\CK \ge 4$, with equality in the normalized centred Gaussian
  case~\eqref{Gaussian} and only in that case.
\end{conj}
%---------------------------------------------------------------------

%%%%%%%%%%%%%%%%%%%%%%%%%%%%%%%%%%%%%%%%%%%%%%%%%%%%%%%%%%%%%%%%%%%%%%
\subsection{Outline of the paper}

In Section~\ref{sec:Gaussian} we prove Theorem~\ref{theo:KPK} in
the simple case of Gaussian potentials. This has a pedagogical
interest but also an interest {\em per se} as the method captures
some (conjectured) optimal
constants. Section~\ref{sec:wittenpoincare} is devoted to
classical results on the Witten-Laplace operator on functions and
a sketch of the proof of Poincar\'e inequalities under
assumptions~\eqref{hyp:intnorm}--\eqref{hyp:regularity}--\eqref{hyp:poincarebasique}
with some short quantitative proofs for which we lack of
references. In Section~\ref{sec:general} we prove
Theorems~\ref{theo:KPK} and~\ref{theo:PKPK} in the general
case. Section~\ref{sec:ao} is devoted to the functional analysis
of operators (Theorem~\ref{theo:ao}) associated with various
quadratic forms under consideration. We prove
Theorem~\ref{theo:KPK0} on zeroth order Korn inequalities in
Section~\ref{sec:kornzero}. Appendix~\ref{Appendix:A} is devoted
to generalizations, a discussion of the measure of the defects of
axisymmetry by rigidity constants, and an elementary application
of our main results to a simple kinetic equation with multiple
conservations laws. For the convenience of the reader, some
computational details are collected in Appendix~\ref{Appendix:B}.

%%%%%%%%%%%%%%%%%%%%%%%%%%%%%%%%%%%%%%%%%%%%%%%%%%%%%%%%%%%%%%%%
%%%%%%%%%%%%%%%%%%%%%%%%%%%%%%%%%%%%%%%%%%%%%%%%%%%%%%%%%%%%%%%%
\section{Proof of the Korn inequalities of Theorem~\ref{theo:KPK}
  in the Gaussian case}
\label{sec:Gaussian}

Inspired by the proof of~\eqref{eq:KornOmega}, we first prove
inequalities~\eqref{eq:WKfull} and~\eqref{eq:WPKfull} of
Theorem~\ref{theo:KPK} for the normalized Gaussian measure, and
establish the optimality of the constants in that case. We begin
with two useful identities valid for a general function
$\phi\in\mathrm W^{2,\infty}_{\mathrm{loc}}(\R^d)$ and any
$u \in C^1_c(\R^d;\R^d)$,
\begin{eqnarray}
  \label{eq:weakKorn}
  &&\|\nablaskew u\|^2+\|(\nabla-\nabla\phi)\cdot u\,\|^2
     =\|\nablasym u\|^2+\int_{\R^d} D^2 \phi : u \otimes u\,
     e^{-\phi} \dd x\,,\\
  \label{eq:weakKorn2}
  &&\|\nablaD u\|^2 \le 2\,\|\nablasym u\|^2+\int_{\R^d}
     D^2 \phi : u \otimes u\,e^{-\phi} \dd x\,.
\end{eqnarray}
Identity~\eqref{eq:weakKorn} is obtained by a simple integration
by parts, a commutation and the Schwarz Theorem
(or~\eqref{eq:Pointwise1} integrated against $e^{-\phi}$),
while~\eqref{eq:weakKorn2} follows from
$|\nablaD u|^2=|\nablasym u|^2+|\nablaskew u|^2$.

\medskip In the remainder of this section, let us focus on the
{\em Gaussian case}~\eqref{Gaussian} such that
\begin{equation*}
  e^{-\phi(x)}=(2\pi)^{-d/2}\,e^{-\frac12|x|^2}
\end{equation*}
is the standard centred normalized Gaussian. This is the only
Gaussian function satisfying hypotheses~\eqref{hyp:intnorm} with
the additional normalization $\seq{D^2 \phi}=\Id$, and it
satisfies~\eqref{hyp:regularity} with $\varepsilon=0$ and
$C_\varepsilon=C_0=d$ and it
satisfies~\eqref{hyp:poincarebasique} with $\CP=1$. We first
recall the following improved version of the Poincar\'e
inequality.
%--------------------------------------------------------------
\begin{lem}[Improved Poincar\'e inequality]
  \label{lem:improvedpoincare}
  Assume~\eqref{Gaussian}. Then for any $u\in H^1$ such that
  $\langle u_i\,x_j \rangle=\langle u_i \rangle=0$ with
  $i,j=1,\ldots,d$, there holds
  \begin{equation}
    \label{eq:improvedpoincare}
    2\,\norm{u}^2 \le\norm{\nablaD u}^2.
  \end{equation}
\end{lem}
%--------------------------------------------------------------

This result is standard: the operator $-\Delta_\phi$ reduces
after conjugation by $e^{-\phi/2}$ to the harmonic oscillator
$P_\phi=-\,\Delta+|x|^2/4-d/2$ which has a discrete spectrum made
of all nonnegative integers (see Section~\ref{sec:wittenpoincare}
for the definition of the operator). The lowest eigenvalue is $0$
with multiplicity $1$ and the first positive eigenvalue is $1$
with multiplicity $d$ and eigenfunctions $x_j$,
$j=1,2,\ldots,d$. Conditions on $u$ amount to the orthogonality
condition to these two eigenspaces, so that $2$ corresponds to
the next eigenvalue. The result follows from the spectral theorem
(see for instance to~\cite[Lemma 2 of Chapter V and
Chapter~8]{RS75} or~\cite{CFKS87}).

\begin{proof}[Proof of Theorem~\ref{theo:KPK} in the Gaussian
  case]
  Since
  $D^s (\langle D^a u \rangle\,x)=D^s(\langle u
  \rangle)=D(\langle u \rangle)=0$ and
  $\P(\langle D^a u \rangle\,x)=\langle D^a u \rangle\,x$, it is
  enough to prove the inequalities for a vector field $u \in H^1$
  such that $\langle \nablaskew u\rangle=0$ and
  $\langle u \rangle=0$. We have to show that
  \begin{equation}\label{GaussianRelations}
    \norm{u}^2 \leq 2 \norm{D^s u}^2\quad \mbox{and}
    \quad \norm{Du}^2 \leq 4 \norm{D^s u}^2.
  \end{equation}
  Let us define the {\em corrected vector field} $v \in H^1$ by
  \begin{equation*}
    v(x) :=u(x)-B\,x\quad\mbox{where}\quad
    B_{ij} :=\langle u_i\,x_j\,\rangle\,,
  \end{equation*}
  and note the elementary property (using that $e^{-\phi} \dd x$
  is Gaussian)
  \begin{equation*}\label{eq:ippgaussian}
    B_{ij}=\langle u_i\,x_j \rangle=
    \int_{\R^d} u_i\,x_j\,e^{-\phi} \dd x=
    \int_{\R^d} \D_j u_i\,e^{-\phi} \dd x=\seq{D u}_{ij}.
  \end{equation*}
  This implies that the matrix $B$ is symmetric since
  $\langle \nablaskew u\rangle=0$ and that $v$ satisfies
  $\langle v_i\,x_j \rangle=\langle v_i \rangle=0$ for all
  $i,j=1, \ldots, d$. We can then apply the improved Poincar\'e
  inequality~\eqref{eq:improvedpoincare} in
  Lemma~\ref{lem:improvedpoincare} to $v$ and get
  \begin{equation*}
    2\left\|v \right\|^2 \le \left\|\nablaD v \right\|^2.
  \end{equation*}
  Using this together with~\eqref{eq:weakKorn2} and
  $D^2 \phi=\text{Id}$, we obtain
  $2\,\|v\|^2 \le\|\nablaD v\|^2 \le 2\,\|\nablasym
  v\|^2+\|v\|^2$ which implies $\|v\|^2 \le 2\,\|\nablasym v\|^2$
  and $\|D v\|^2 \le 4\,\|\nablasym v\|^2$,
  \ie,~\eqref{GaussianRelations} written for $v$. Next we compute
  \begin{align*}
    &\|D u\|^2=\|D v+B\|^2=\|D^s v+B\|^2+\|D^a v\|^2 =
      \|D v\|^2+|B|^2+2 \int_{\R^d} \left( D^s v : B \right)\,
      e^{-\phi} \dd x\,,\\
    &\|D^s u\|^2=\|D^s v+B\|^2=\|D^s v\|^2+|B|^2 +
      2 \int_{\R^d} \left( D^s v : B \right)\,e^{-\phi} \dd x\,,\\
    &\|u\|^2=\|v+B\,x\|^2=\|v\|^2+|B|^2 +
      2 \int_{\R^d} \left( v\cdot B\,x \right)\,e^{-\phi} \dd x=
      \|v\|^2+|B|^2,
  \end{align*}
  where we used the fact that $(D^s u+B)$ and $D^a u$ are
  orthogonal in $L^2$ and $\langle v_i\,x_j \rangle=0$. By an
  integration by parts, we obtain
  \begin{equation*}\label{eq:cross-nul}
    \int_{\R^d} \left( \nablasym v : B \right)\,e^{-\phi} \dd x=
    \frac12 \sum_{i,j} B_{ij} \int_{\R^d} (\partial_i v_j+\partial_j
    v_i)\,
    e^{-\phi} \dd x=
    \frac12 \sum_{i,j} B_{ij} \int_{\R^d}
    \left(x_i\,v_j+x_j\,v_i\right )\,e^{-\phi} \dd x=0
  \end{equation*}
  using again $\langle v_i\,x_j \rangle=0$. Altogether, we deduce
  that
  \begin{align*}
    \|D u\|^2=\|D v\|^2+|B|^2,\quad\|D^s u\|^2=\|D^s v\|^2+|B|^2,
    \quad\|u\|^2=\|v\|^2+|B|^2.
  \end{align*}
  We deduce~\eqref{GaussianRelations} on $u$
  from~\eqref{GaussianRelations} on $v$ and the last equations,
  which proves~\eqref{eq:WKfull} with $\CK\le4$
  and~\eqref{eq:WPKfull} with $\CPK\le2$.  To
  saturate~\eqref{GaussianRelations} it is enough to search for
  $u=v$ with $B=0$.  With
  $u(x)=(1-x_2^2,x_1\,x_2,0,\ldots,0)^\perp$, an elementary
  computation (see details in Appendix~\ref{Appendix:B5}) shows
  that
  $$
  \seq u=0,\quad\|u\|^2=3\,,\quad\P(u)=0,\quad\fP
  (Du)=0=\seq{D^au}\,,\quad\|D^su\|^2=\frac32\,,\quad\|D^au\|^2=\frac92\,,\quad\|Du\|^2=6\,.
  $$
  This completes the proof of~\eqref{eq:WKfull} with $\CK=4$
  and~\eqref{eq:WPKfull} with $\CPK=2$.

  It remains to establish~\eqref{eq:WPKstrong}. By expanding the
  square $\int_{\R^d}|D(u\,e^{-\phi/2})|^2 \dd x$ as
  in~\cite[ineq.~(4)]{Dolbeault_2012}, we obtain after one
  integration by parts that
  $$
  \int_{\R^d}|x|^2\,|u(x)|^2\,e^{-\phi(x)}\dd x\le4\int_{\R^d}|D
  u|^2\,e^{-\phi}\dd x+2\,d\int_{\R^d}|u|^2\,e^{-\phi}\dd x\,.
  $$
  Combined with~\eqref{eq:WKfull} and~\eqref{eq:WPKfull}, this
  completes the proof with $\CPK\le\CSPK\le2\,(2\,d+9)$.
\end{proof}

%%%%%%%%%%%%%%%%%%%%%%%%%%%%%%%%%%%%%%%%%%%%%%%%%%%%%%%%%%%%%%%%%%%%%%
%%%%%%%%%%%%%%%%%%%%%%%%%%%%%%%%%%%%%%%%%%%%%%%%%%%%%%%%%%%%%%%%%%%%%%
\section{The Witten-Laplace operator on scalar functions and Poincar\'e inequalities}\label{sec:wittenpoincare}

Here we consider the Poincar\'e inequalities of
Proposition~\ref{prop:poincares} and some related properties of
the Witten-Laplace operator $\Delta_\phi$, as defined
in~\eqref{def:Deltaphi}, in the case of a general probability
measure $e^{-\phi}\dd x$ with a potential~$\phi$ such that
assumptions~\eqref{hyp:intnorm}--\eqref{hyp:regularity}--\eqref{hyp:poincarebasique}
are fulfilled. Some results of this section are classical and we
claim no originality. For a general theory of self-adjoint
operators, we refer for instance to~\cite{Sim78,RS75} and we
refer to~\cite{Wit82,HS94,Sjo96,Joh00} or~\cite{HN05} for more
details on Witten-Laplace operators. Proofs are given when we are
not aware of any precise reference or when we look for explicit
estimates.

%%%%%%%%%%%%%%%%%%%%%%%%%%%%%%%%%%%%%%%%%%%%%%%%%%%%%%%%%%%%%%%%%%%%%%
\subsection{Two toolboxes and the proof of the strong Poincar\'e
  inequality
  (Proposition~\texorpdfstring{\ref{prop:poincares}}{Proposition5})}\label{Sec:Toolboxes}

For all functions $f$, $g\in \ccc_c^\infty(\R^d; \R)$, we have by
integration by parts
\begin{equation*}
  \label{eq:ipp}
  (-\Delta_\phi f, g)=-\,\sep{\nablaphi\cdot \nabla f, g}=
  (\nabla f, \nabla g)\,,
\end{equation*}
so that $-\nablaphi=-\,\nabla+\nabla\phi$ is the formal adjoint
of $\nabla$, $-\Delta_\phi$ is nonnegative and symmetric, and
$\Lambda$, as defined by~\eqref{Lambda}, is symmetric. The
Lax-Milgram theorem allows us to solve in $H^1$, equipped with
the norm $f \mapsto (\norm{f}^2+\norm{\nabla f}^2)^{1/2}$, the
problem $\Lambda f=\xi$ for any given $\xi \in H^{-1}$, and to
build a self-adjoint extension of $\Lambda$ associated to the
coercive bilinear form
$(f,g) \mapsto (\nabla f, \nabla g)+(f,g)$. On the other hand, by
the well-known change of function $f\mapsto e^{-\phi/2}f$,
$\Lambda$ is conjugated~to
\begin{equation*}
  \label{eq:conjugaison}
  P_\phi :=e^{-\phi/2}\,\Lambda\,e^{-\phi/2} =
  -\,\Delta+\tfrac14\,|\nabla\phi|^2-\tfrac12\,\Delta \phi+1
\end{equation*}
acting on the usual space $L^2(\dd
x)$. From~\eqref{hyp:regularity}, we get that
$|\nabla\phi|^2/4-\Delta \phi/2$ is bounded from below. From
Kato's result~\cite{Kat72} (also see, \eg,~\cite[Theorem
X-28]{RS75}), this implies that $\Lambda$ has a unique Friedrichs
self-adjoint extension such that $\ccc_c^\infty(\R^d;\R)$ is
dense in its domain w.r.t.~the graph norm, that is, $\Lambda$~is
{\em essentially self-adjoint}. For notational simplicity, we use
the same name for the operator and for its extension. We denote
by $\ddd(\Lambda)$ the domain of $\Lambda$.

Hence $\ccc_c^\infty(\R^d;\R)$ is a core for the self-adjoint
operator $\Lambda \geq \Id$, which has a one-to-one operator
extension from $H^1$ to $H^{-1}$. Tools of functional calculus
and spectral analysis apply. This gives sense to $\Lambda^\sigma$
with domain $\ddd(\Lambda^\sigma)$ for all $\sigma\in \R$. For
instance, $\ddd(\Lambda^{1/2})=H^1$ and $\Lambda^{1/2}$ has a
bounded one-to-one operator extension from $L^2$ to $H^{-1}$ by
duality. Recall that no specific growth, apart from the general
condition~\eqref{hyp:regularity}, is assumed on $|\nabla\phi|$ at
infinity in the computations of this section
(see~\cite{HN04,HN05,MR913672} for other results without growth
condition). Let us show that~\eqref{hyp:regularity} implies that
$\wgt\,f$ is square integrable whenever $f\in H^1$, which allows
to make sense of $\norm{\nabla\phi\cdot u}$ in
inequalities~\eqref{eq:WPK} and~\eqref{eq:WK}. 
%--------------------------------------------------------------
\begin{prop}[$H^1$ toolbox]
  \label{prop:H1}
  Assume~\eqref{hyp:regularity}. Then the space $H^1$ is
  \begin{equation*}
    H^1 =
    \set{f\in L^2 \ : \ \nabla f \in L^2\ \mbox{{\normalfont and}}
      \ \wgt\,f \in L^2}
  \end{equation*}
  and for any $f \in L^2$, we have the inequalities
  \begin{equation}\label{eq:toolboxH1}
    \big\|\nabla\Lambda^{-1/2}\,f\big\|^2 \leq\|f\|^2,\quad
    \big\|\wgt\,\Lambda^{-1/2}\,f\big\|^2 \leq \Cphi\,\|f\|^2,
  \end{equation}
  \begin{equation}\label{eq:toolboxH1adj}
    \big\|\Lambda^{-1/2}\,\nabla f\big\|^2 \leq\|f\|^2,\quad
    \big\|\Lambda^{-1/2}\,\wgt\,f\big\|^2 \leq \Cphi\,\|f\|^2.
  \end{equation}
\end{prop}
%--------------------------------------------------------------

\begin{proof}

  For all $f \in D(\Lambda)$ we have
  $\norm{\nabla f}^2 \leq (\Lambda
  f,f)=\|\Lambda^{1/2}\,f\|^2$. By density of $D(\Lambda)$ in
  $H^1$ we get $\norm{\nabla f}^2 \leq\|\Lambda^{1/2}\,f\|^2$ for
  all $f \in H^1$ and applying this inequality to
  $\Lambda^{-1/2}\,f \in H^1$ proves the first inequality
  in~\eqref{eq:toolboxH1}.

  Let us note that
  $0\leq|\nabla\phi|^2-4\,\sqrt
  d\,|D^2\phi|+\Cphi-1\leq|\nabla\phi|^2-4\,\Delta\phi+\Cphi-1$
  because $\Delta\phi\le\sqrt d\,|D^2\phi|$ and according
  to~\eqref{eq:cphi}, so that
  \begin{equation*}
    \wgt^2 \leq
    8\(\tfrac14\,|\nabla\phi|^2-\tfrac12\,\Delta \phi\)+\Cphi\,.
  \end{equation*}
  As a consequence, we get the operator inequality
  $\wgt^2 \leq -\,8\,\Delta_\phi+\Cphi\,\Id \leq \Cphi\,\Lambda$
  using the fact that the usual Laplacian $-\Delta$ is
  nonnegative on $L^2(\dd x)$ and $\Cphi\ge8$. This implies that,
  for all $f \in D(\Lambda)$, we have that $\wgt\,f$ is in $L^2$
  and
  $\norm{\wgt\,f}^2 \leq \Cphi\,(\Lambda
  f,f)=\Cphi\,\|\Lambda^{1/2}\,f\|^2$. By density of $D(\Lambda)$
  in $H^1$, we~get
  \begin{equation}\label{eq:phimajparnabla}
    \forall\,f \in H^1,\quad\big\|\wgt\,f\big\|^2 \leq
    \Cphi\,\big\|\Lambda^{1/2}\,f\big\|^2.
  \end{equation}
  For any $f\in L^2$, applying~\eqref{eq:phimajparnabla} to
  $\Lambda^{-1/2}\,f \in H^1$ gives the second inequality
  in~\eqref{eq:toolboxH1}. Inequalities
  in~\eqref{eq:toolboxH1adj} are obtained
  from~\eqref{eq:toolboxH1} by considering the adjoint operators.
\end{proof}

\begin{proof}[Proof of the strong Poincar\'e
  inequality~\eqref{eq:strongpoincare}] So far we did not
  use~\eqref{hyp:poincarebasique} and its spectral
  consequences. Using the density of $\ccc_c^\infty(\R^d,\R)$ in
  $D(\Lambda)$ and~\eqref{hyp:poincarebasique}, we get that $0$
  is an isolated eigenvalue of $-\Delta_\phi=\Lambda-1$ with
  associated eigenspace
  $\R$. Inequality~\eqref{eq:strongpoincare} follows
  from~\eqref{eq:phimajparnabla} applied to $f-\seq{f}$
  and~\eqref{hyp:poincarebasique}, with
  $\CSP \leq \Cphi\,(1+\CP)$.
\end{proof}

The following toolbox is a key step in proof of the {\em
  Poincar\'e-Lions} inequality~\eqref{eq:poincarelions}.
%-------------------------------------------------------------
\begin{prop}[$D(\Lambda)$-Toolbox]
  \label{prop:domain}
  Assume~\eqref{hyp:intnorm} and~\eqref{hyp:regularity}. Then
  \begin{equation*}
    \ddd(\Lambda)=\Big\{ f\in L^2\,:\,\big\|\wgt^2\,f\big\| +
    \big\|\wgt\,\nabla f\big\|+\big\|D^2f\big\|<+\infty \Big\}\,
  \end{equation*}
  and there exists a positive constant $\CB$ depending only on
  $\Cphi$, $\Cphi'$ and $d$ such that, for any $f\in L^2$,
  \begin{equation}\label{eq:toolbox}
    \big\|D^2 \Lambda^{-1} f \big\|^2+
    \big\|\wgt\,\nabla\Lambda^{-1} f \big\|^2+
    \big\|\wgt^2\,\Lambda^{-1} f \big\|^2 \leq \CB\,\|f\|^2,
  \end{equation}
  \begin{equation}\label{eq:toolboxadj}
    \big\|\Lambda^{-1} D^2f \big\|^2+
    \big\|\Lambda^{-1} \wgt\,\nabla f \big\|^2+
    \big\|\Lambda^{-1} \wgt^2\,f \big\|^2 \leq \CB\,\|f\|^2.
  \end{equation}
\end{prop}
%-------------------------------------------------------------
\begin{proof} Inequality~\eqref{eq:toolboxadj} follows
  from~\eqref{eq:toolbox} by duality
  using~\eqref{hyp:regularity}.

  Let us denote by $\mathcal S$ the subspace of $L^2$ such that
  $D^2f$, $\wgt\,\nabla f$ and $\wgt^2\,f$ are square
  integrable. It is elementary to check that
  $\ddd(\Lambda)\subset\mathcal S$. In order to prove that,
  reciprocally, $\mathcal S\subset\ddd(\Lambda)$, let us argue by
  density of $\ccc_c^\infty(\R^d;\R)$. For any
  $f \in \ccc_c^\infty(\R^d;\R)$, let us prove that
  $\xi=\Lambda f=-\,\Delta_\phi f+f$ is such that
  \begin{equation}
    \label{eq:Phi1estimate}
    \big\|\wgt^2f\big\|^2+\big\|\wgt\,\nabla
    f\big\|^2+\big\|D^2f\big\|^2\leq \CB\,\|\xi\|^2
  \end{equation}
  for some explicit constant $\CB$, so that
  $\ddd(\Lambda)=\mathcal S$ and~\eqref{eq:toolbox} directly
  follow.

  It follows from
  $\norm{f}^2 \leq (\Lambda f,f) \leq \norm{\Lambda f}\,\norm{f}$
  that $\|f\|\leq\|\xi\|$. Similarly, using~\eqref{eq:toolboxH1},
  we have
  \begin{equation}
    \label{eq:Dphifxi}
    \|\wgt\,f\|\leq 
    \Cphi^{1/2}\, \| \Lambda^{1/2} f\| = \Cphi^{1/2}\, (\Lambda
    f,f)^{1/2} \le \Cphi^{1/2}\,\|\xi\|\,.
  \end{equation}
  Next, we estimate
  $\big\|\wgt^2\,f\big\|$. Using~\eqref{eq:toolboxH1} and the
  triangular inequality we have
  \begin{equation*}
    \|\wgt^2\,f\|=\|\wgt\,\wgt\,f\|\leq
    \Cphi^{1/2}\,\|\Lambda^{1/2}\,(\wgt\,f)\|
    \leq \Cphi^{1/2}\,\norm{\nabla(\wgt\,f)}
    +\Cphi^{1/2}\,\norm{\wgt\,f}\,.
  \end{equation*}
  With $\nabla(\wgt\,f)=\wgt\,\nabla f+(\nabla\wgt)\,f$, we get
  \begin{equation*}
    \begin{split}
      \|\wgt^2\,f\|
      & \leq \Cphi^{1/2}\,\norm{\wgt\,\nabla f}+\Cphi^{1/2}\,
      \norm{\wgt\,f}+\Cphi^{1/2}\,\norm{(\nabla\wgt)\,f}\\
      & \leq \Cphi^{1/2}\,\norm{\wgt\,\nabla
        f}+\Cphi\,\norm{\xi}+\Cphi^{1/2}\,\norm{(\nabla\wgt)\,f}
    \end{split}
  \end{equation*}
  by~\eqref{eq:Dphifxi}. Estimate~\eqref{eq:cphi} yields
  \begin{equation}
    \label{Pointwise:nablaphiter}
    \big|\nabla\wgt\big|\le\frac{|D^2\phi|\,|\nabla\phi|}{\wgt}\le
    \tfrac14\,\Cphi^{-1/2}\,\wgt^2+\tfrac14\,\Cphi'\,\Cphi^{-1/2}
  \end{equation}
  and, as a consequence,
  \begin{equation*}
    \begin{split}
      & \big\|\wgt^2\,f\big\|\leq \Cphi^{1/2}\,\norm{\wgt\,\nabla
        f}+\Cphi\,\norm{\xi}+\tfrac14\,\norm{\wgt^2\,f}
      +\tfrac14\,\Cphi' \norm{f}\,,
    \end{split}
  \end{equation*}
  so that using $\Cphi' \geq \Cphi$ and
  $\norm{f} \leq \norm{\xi}$, we get
  \begin{equation}\label{DVDu}
    \big\|\wgt^2\,f\big\|\leq \tfrac43\,\Cphi^{1/2}\,
    \norm{\wgt\,\nabla f}+\tfrac53\,\Cphi' \norm{\xi}\,.
  \end{equation}

  Next we estimate $\big\|\wgt\,\nabla f\big\|^2$ by
  \begin{equation*}
    \label{Eqn:intermCite}
    \begin{split}
      \big\|\wgt\,\nabla f\big\|^2
      &=( \wgt^2\,\nabla f, \nabla f)=
      ( \nabla (\wgt^2\,f), \nabla f)
      - ( (\nabla\wgt^2)\,f, \nabla f)\\
      &=( \wgt^2\,f, \xi-f)- ( (\nabla\wgt^2)\,f, \nabla f)\\
      &\leq ( \wgt^2\,f, \xi)
      - 2\,\big( (\nabla\wgt)\,f, \wgt\,\nabla f\big)\\
      & \leq
      \norm{\wgt^2\,f}\,\norm{\xi}
      +2\,\norm{(\nabla\wgt)\,f}\,\norm{\wgt\,\nabla
        f}\,.
    \end{split}
  \end{equation*}
  Using~\eqref{Pointwise:nablaphiter} and
  $\norm{f} \leq \norm{\xi}$, we get
  \begin{equation}
    \label{DVDubis}
    \big\|\wgt\,\nabla f\big\|^2\leq
    \norm{\wgt^2\,f}\,\norm{\xi}+\tfrac12\,\Cphi^{-1/2}\,
    \norm{\wgt^2\,f} \,\norm{\wgt\,\nabla
      f}+\tfrac12\,\Cphi'\,\Cphi^{-1/2}\,\norm{\xi}\,
    \norm{\wgt\,\nabla
      f}\,.
  \end{equation}
  With elementary estimates, we deduce
  from~\eqref{DVDu}-\eqref{DVDubis} that
  \begin{equation}
    \label{Est:4}
    \big\|\wgt\,\nabla f\big\|\le 9\,\Cphi'\,\|\xi\|\quad
    \mbox{and}\quad \big\|\wgt^2f\big\|\leq 14\,\Cphi'\,\|\xi\|\,.
  \end{equation}

  Integrations by parts show that
  \begin{align*}
    \big\|D^2f\big\|^2
    &={\textstyle \sum_{i,j}}(\partial_{ij}f, \partial_{ij}
      f)={\textstyle \sum_{i,j}} \big(\partial_j f,
      -\,\partial_{iij} f+\partial_{ij}f\,\partial_i \phi\big)\\
    &={\textstyle \sum_{i,j}}\big(\partial_j f, \partial_j
      (-\,\partial_{ii} f)\big)+\tfrac12\,{\textstyle
      \sum_{i,j}}\(\partial_i( |\partial_j f|^2), \partial_i
      \phi\) =\big(\nabla f,\nabla(-\Delta
      f)\big)+\tfrac12\(|\nabla
      f|^2,|\nabla\phi|^2-\Delta\phi\).
  \end{align*}
  Using the elementary estimates
  \begin{align*}
    &\big(\nabla f,\nabla(-\Delta f)\big)=(\Delta_\phi f,\Delta
      f)=(f-\xi,\Delta f)\le(f,\Delta f)+\|\xi\|\,\|\Delta f\|\,,\\
    &(f,\Delta f)=(f,\Delta_\phi f)
      +(f,\nabla\phi\cdot\nabla f)=-\,\|\nabla f\|^2
      +\tfrac12\, (\nabla f^2,\nabla\phi)
      \le-\tfrac12\,(f^2,\Delta_\phi\phi)
      =\tfrac12\(|f|^2,|\nabla\phi|^2-\Delta\phi\),
  \end{align*}
  and using~\eqref{eq:cphi} and $\wgt\ge1$ and the fact that
  \begin{equation*}
    |\nabla\phi|^2-\Delta\phi\le|\nabla\phi|^2+\sqrt
    d\,|D^2\phi|\le|\nabla\phi|^2
    +\tfrac14\(|\nabla\phi|^2+C_\phi-1\)
    \le\tfrac14\(5\,\wgt^2+C_\phi-6\)\le\tfrac14\(C_\phi-1\)\wgt\,,
  \end{equation*}
  we obtain, using also~\eqref{Est:4}, the estimate
  \begin{multline*}
    \frac1d\,\|\Delta f\|^2\le\,\big\|D^2f\big\|^2\le\|\xi\|\,
    \|\Delta f\|+\tfrac12\(|f|^2+|\nabla f|^2,|\nabla\phi|^2-
    \Delta\phi\)\\
    \le\|\xi\|\,\|\Delta
    f\|+\tfrac18\(C_\phi-1\)\(\big\|\wgt^2f\big\|^2
    +\big\|\wgt\,\nabla
    f\big\|^2\)\le\|\xi\|\,\|\Delta f\|+C\,\|\xi\|^2
  \end{multline*}
  with $C=\tfrac18\,277\(C_\phi-1\){C_\phi'}^2$ because
  $277=9^2+14^2$. As a straightforward consequence, we obtain
  \begin{equation*}
    \|\Delta
    f\|\le\tfrac12\(d+\sqrt{d^2+4\,C}\)\|\xi\|
    \quad\mbox{and}\quad\big\|D^2f\big\|^2\le\tfrac
    Cd+\tfrac12\(d+\sqrt{d^2+4\,C}\)\|\xi\|^2.
  \end{equation*}
  With~\eqref{Est:4}, this completes the proof
  of~\eqref{eq:Phi1estimate}. A detailed computation of $\CB$ is
  given in Appendix~\ref{Sec:Dlambda}.
\end{proof}

%%%%%%%%%%%%%%%%%%%%%%%%%%%%%%%%%%%%%%%%%%%%%%%%%%%%%%%%%%%%%%%%%%%%%%
\subsection{The Poincar\'e-Lions inequality (Proposition~\texorpdfstring{\ref{prop:poincares}}{Proposition5})}\label{Sec:PL}

We now focus on~\eqref{eq:poincarelions}. As a preliminary remark, note that this inequality is the counterpart in the whole space of the so-called {\em Lions lemma} in the smooth bounded domain case $\Omega \subset \R^d$, which amounts to the existence of some $c_\Omega >0$ such that
\begin{equation*}
\forall\,f\in L^2(\Omega)\,,\quad c_\Omega\,\|f - \left<f\right>\|^2_{L^2(\Omega)} \leq\|\nabla f\|^2_{H^{-1}(\Omega)} \le d\,\|f-\left<f\right>\|^2_{L^2(\Omega)}
\end{equation*}
(see for instance~\cite{MR0521262} and~\cite[Theorem 6.11.4]{Cia13}). This inequality belongs to the folklore in Hodge theory, see for instance~\cite{HS94} or~\cite{Joh00}, with variants involving the so-called {\em Witten-Laplacian on one-forms}.

\begin{proof}[Proof of the Poincar\'e-Lions inequality~\eqref{eq:poincarelions}]
First note that the right inequality directly follows from~\eqref{eq:toolboxH1adj} applied to $f-\seq{f}$. We focus on the left one. The spectral theorem implies for all $f \in D(\Lambda)$ with $\seq{f}=0$
\begin{equation}\label{eq:spectralthm}
(1+\CP)^{-1}\,\norm{f}^2 \leq \sep{(-\Delta_\phi)\,\Lambda^{-1} f, f}=\sep{\Lambda^{1/2}\,\nabla\Lambda^{-1} f, \Lambda^{-1/2}\,\nabla f}\le\|\Lambda^{1/2}\,\nabla\Lambda^{-1} f\|\,\|\Lambda^{-1/2}\,\nabla f\|
\end{equation}
because $1/(1+\CP)\le s/(s+1)$ for any $s\in[1/\CP,\infty)$. Let us prove that $\Lambda^{1/2}\,\nabla\Lambda^{-1}$ is a bounded operator. Using the commutator $[\Lambda,\nabla]=-\,D^2\phi\,\nabla$, we compute
\begin{multline*}
\Lambda^{1/2}\,\nabla\Lambda^{-1} =\Lambda^{-1/2}\,\Lambda \nabla\Lambda^{-1}= \Lambda^{-1/2}\,\nabla+\Lambda^{-1/2}\,[\Lambda, \nabla]\,\Lambda^{-1}\\
=\Lambda^{-1/2}\,\nabla-\Lambda^{-1/2}\,D^2 \phi\,\nabla\Lambda^{-1}=\Lambda^{-1/2}\,\nabla+\Lambda^{-1/2}\,\wgt\,\big(\wgt^{-1} D^2\phi\,\wgt^{-1}\big)\,\wgt\,\nabla\Lambda^{-1}\,.
\end{multline*}
{}From~\eqref{eq:toolboxH1adj}, we know that $\Lambda^{-1/2}\,\nabla$ and $\Lambda^{-1/2}\,\wgt$ are bounded respectively by $1$ and $\sqrt{\Cphi}$, from~\eqref{eq:toolbox} the operator $\wgt\,\nabla\Lambda^{-1}$ is bounded by $\sqrt{\CB}$, and we have
$$
\wgt^{-1} D^2\phi\,\wgt^{-1}\le\frac{\Cphi'}{4\,\sqrt{\Cphi}}
$$
as a consequence of~\eqref{eq:cphi}. Altogether, $\Lambda^{1/2}\,\nabla\Lambda^{-1}$ is bounded and
\begin{equation}\label{eq:estimatecpl1new}
\|\Lambda^{1/2}\,\nabla\Lambda^{-1}f\|\leq\(1+\tfrac14\,\Cphi'\,\sqrt{\CB}\) \norm f\,,
\end{equation}
which completes the proof with
\be\label{CPL}
\CPL=(1+\CP)^2\,\big(1+\tfrac14\,\Cphi'\,\sqrt{\CB}\big)^2.
\ee
\end{proof}

%%%%%%%%%%%%%%%%%%%%%%%%%%%%%%%%%%%%%%%%%%%%%%%%%%%%%%%%%%%%%%%%%%%%%%
%%%%%%%%%%%%%%%%%%%%%%%%%%%%%%%%%%%%%%%%%%%%%%%%%%%%%%%%%%%%%%%%%%%%%%
\section{Proof of the Korn inequalities of
  Theorems~\ref{theo:KPK} and~\ref{theo:PKPK} for general
  potentials}
\label{sec:general}

In this section, we assume that the potential
satisfies~\eqref{hyp:intnorm},~\eqref{hyp:regularity}
and~\eqref{hyp:poincarebasique}.

%%%%%%%%%%%%%%%%%%%%%%%%%%%%%%%%%%%%%%%%%%%%%%%%%%%%%%%%%%%%%%%
\subsection{Proof of Theorem~\ref{theo:KPK}}
\label{Sec:Prf1}

As a preliminary remark, we recall that
\begin{equation*}
  \label{eq:pdudau}
  \forall\,u \in H^1,\quad \fP(Du)=\seq{\nablaskew u}.
\end{equation*}
Indeed $Du=(Du- \seq{\nablaskew u})+\seq{\nablaskew u}$ is an
orthogonal decomposition because
\begin{equation*}
  (Du- \seq{\nablaskew u}, \seq{\nablaskew u})=
  (\nablaskew u- \seq{\nablaskew u}, \seq{\nablaskew u})
  +(\nablasym u, \seq{\nablaskew u})=
  \seq{\nablasym u}:\seq{\nablaskew u}=0
\end{equation*}
and the uniqueness of this decomposition shows the result.
\smallskip

\noindent$\rhd$ {\em Proof of~\eqref{eq:WKfull}}. Let us take
$u\in H^1$ such that $\seq{u}=0$ and $\seq{\nablaskew
  u}=0$. Using the Poincar\'e-Lions
inequality~\eqref{eq:poincarelions}, we have
\begin{equation}
  \label{eq:WKfull1}
  \|Du\|^2=\|\nablasym u\|^2+\|\nablaskew u\|^2 =
  \|\nablasym u\|^2+\sum_{i,j=1}^d\|(\nablaskew u)_{ij}\|^2
  \leq\|\nablasym u\|^2+\CPL \,
  \|\Lambda^{-1/2}\, \nabla (\nablaskew u)\|^2
\end{equation}
with
$\|\Lambda^{-1/2}\,\nabla (\nablaskew
u)\|^2=\sum_{i,j=1}^d\|\Lambda^{-1/2}\,\nabla (\nablaskew
u)_{ij}\|^2$. The Schwarz Theorem as stated in~\eqref{Schwarz}
gives
\begin{equation}
  \label{eq:WKfull2}
  \|\Lambda^{-1/2}\,\nabla (\nablaskew
  u)\|^2\leq 2 \sum_{i,j,k=1}^d \sep{\|\Lambda^{-1/2}\,\D_i
    (\nablasym u)_{jk}\|^2+\|\Lambda^{-1/2}\,\D_j (\nablasym
    u)_{ik}\|^2} =4 \sum_{j,k=1}^d\|\Lambda^{-1/2}\,\nabla
  (\nablasym u)_{jk}\|^2.
\end{equation}
The right-hand side of the Poincar\'e-Lions
inequality~\eqref{eq:poincarelions} yields
\begin{equation*}
  \sum_{j,k=1}^d\|\Lambda^{-1/2}\,\nabla (\nablasym u)_{jk}\|^2
  \leq \sum_{j,k=1}^d\|(\nablasym u)_{jk}\|^2=\|\nablasym u\|^2.
\end{equation*}
Together with~\eqref{eq:WKfull1} and~\eqref{eq:WKfull2}, this
gives $\|Du\|^2 \leq (1+4\,\CPL)\,\|\nablasym u\|^2$ so that we
can take $\CK \leq 1+4\,\CPL$. This proves~\eqref{eq:WKfull}
since $\nablasym \RR=\set{0}$.  \smallskip

\noindent
$\rhd$ {\em Proof of~\eqref{eq:WPKfull}}. Let us take $u\in H^1$
such that $\seq{u}=0$ and $\P(u)=0$. By definition of $\P$, we
have $ \norm{u}^2 \leq \norm{u- \fP(Du)\,x} $ since
$x \mapsto \fP(Du)\,x$ is in
$\RR$. Applying~\eqref{hyp:poincarebasique} and~\eqref{eq:WKfull}
gives
\begin{equation*}
  \norm{u}^2 \leq \CP \norm{Du- \fP(Du)}^2
  \leq \CP\,\CK\,\|\nablasym u\|^2
\end{equation*}
This proves~\eqref{eq:WPKfull} with $\CPK \leq \CP\,\CK$.
\smallskip

\noindent $\rhd$ {\em Proof of~\eqref{eq:WPKstrong}}. Let us take
$u\in H^1$ such that $\seq{u}=0$ and $\P(u)=0$. Applying the
strong Poincar\'e inequality~\eqref{eq:strongpoincare} and the
Korn inequality~\eqref{eq:WKfull} gives
\begin{equation}
  \label{eq:WPKstrong1}
  \norm{\wgt\,u}^2 \leq \CSP \norm{Du}^2=
  \CSP \sep{ \norm{Du-\fP(Du)}^2+\norm{\fP(Du)}^2}
  \leq \CSP\,\CK \norm{\nablasym u}^2+\CSP \norm{\fP(Du)}^2.
\end{equation}
An integration by parts, Jensen's inequality and the
Cauchy-Schwarz inequality show that
\begin{equation}\label{eq:fpbounded}
  \norm{\fP(Du)}^2=|\seq{\nablaskew u}|^2 =
  \tfrac14\,\sum_{i,j=1}^d\abs{ \int_{\R^d} \big(\D_j \phi\,u_i-
    \D_i \phi\,u_j\big)\,e^{-\phi} \dd x }^2
  \leq \norm{\nabla\phi}^2 \norm{u}^2.
\end{equation}
An integration by parts, $\Delta\phi\le\sqrt d\,|D^2\phi|$
and~\eqref{eq:cphi} provide us with
\begin{equation*}
  \int_{\R^d} |\nabla\phi|^2\,e^{-\phi} \dd x =
  \int_{\R^d} \Delta \phi\,e^{-\phi} \dd x \leq
  \tfrac14\int_{\R^d} |\nabla\phi|^2\,e^{-\phi} \dd x +
  \tfrac14\(\Cphi-1\)
\end{equation*}
so that $\norm{\nabla\phi}^2 \leq 3\,\Cphi$ and we conclude that
$\norm{\fP(Du)}^2 \leq 3\,\Cphi\,\norm{u}^2 \leq
3\,\Cphi\,\CPK\,\|\nablasym u\|^2$
by~\eqref{eq:WPKfull}. Inserting this estimate
in~\eqref{eq:WPKstrong1} completes the proof
of~\eqref{eq:WPKstrong} with
$\CSPK \leq \CSP(\CK+3\,\Cphi\,\CPK)$.\qed

%%%%%%%%%%%%%%%%%%%%%%%%%%%%%%%%%%%%%%%%%%%%%%%%%%%%%%%%%%%%%%%
\subsection{Proof of Theorem~\ref{theo:PKPK}}
\label{Sec:Prf2}~
\smallskip

\noindent$\rhd$ {\em Proof of~\eqref{eq:WPK}}. Since for any
$R\in\RR_\phi$, $D^sR=0$ and $\nabla\phi\cdot R=0$, we can
consider $u\in H^1$ such that $\P_\phi(u)=0$ without loss of
generality, so that $\P(u) \in \RR_\phi^c$. According
to~\eqref{eq:WPKfull} and by definition of the rigidity
constant~$\CRV$ in~\eqref{eq:rigidityvect}, we have
\begin{multline*}
  \norm{u}^2=\norm{u-\P(u)-\seq{u}}^2+\norm{\P(u)+\seq{u}}^2
  \leq \CPK\,\|\nablasym u\|^2 +
  \CRV \norm{\nabla\phi\cdot (\P(u)+\seq{u})}^2\\
  \leq \CPK\,\|\nablasym u\|^2+2\,\CRV \norm{\nabla\phi\cdot
    u}^2+2\,\CRV \norm{\nabla\phi\cdot (u-\P(u)-\seq{u})}^2.
\end{multline*}
Applying then the strong Poincar\'e-Korn
inequality~\eqref{eq:WPKstrong} gives
\begin{equation*}
    \norm{u}^2 \leq \CPK\,\|\nablasym u\|^2+2\,\CRV
    \norm{\nabla\phi\cdot u }^2+2\,\CRV\,\CSPK\,\|\nablasym
    u\|^2.
\end{equation*}
This completes the proof of~\eqref{eq:WPK} with
$\CPK' \leq \CPK+2\,\CRV\,\CSPK$.
\smallskip

\noindent $\rhd$ {\em Proof of~\eqref{eq:WK}}. Since for any
$R\in\RR_\phi$, $D^sR=0$ and $\nabla\phi\cdot R=0$, we can again
consider $u\in H^1$ such that $\fP_\phi(Du)=0$ without loss of
generality, so that $\fP(Du) \in \fM_\phi^c$. According
to~\eqref{eq:WKfull} and by definition of the rigidity
constant~$\CRD$ in~\eqref{eq:rigiditydiff}, we have
\begin{multline*}
  \norm{Du}^2=\norm{Du-\fP(Du)}^2+\norm{\fP(Du)}^2 \leq
  \CK\,\|\nablasym u\|^2+\CRD \norm{\nabla\phi\cdot
    (\fP(Du)\,x+\seq{u})}^2 \\
  \leq \CK\,\|\nablasym u\|^2+2\,\CRD \norm{\nabla\phi\cdot
    u}^2+2\,\CRD \norm{\nabla\phi\cdot (u-\fP(D
    u)\,x-\seq{u})}^2.
\end{multline*}
Applying the strong Poincar\'e
inequality~\eqref{eq:strongpoincare} gives
\begin{equation*}
  \norm{Du}^2 \leq \CK\,\|\nablasym u\|^2+2\,\CRD
  \norm{\nabla\phi\cdot u }^2+2\,\CRD\,\CSP\|Du- \fP(Du)\|^2,
\end{equation*}
and by the Korn inequality~\eqref{eq:WKfull} again,
\begin{equation*}
  \norm{Du}^2 \leq \CK\,\|\nablasym u\|^2+2\,\CRD
  \norm{\nabla\phi\cdot u }^2+2\,\CRD\,\CSP\,\CK\,\|\nablasym
  u\|^2,
\end{equation*}
This gives~\eqref{eq:WPK} with $\CK' \leq \CK( 1+2\,\CRD\,\CSP)$,
with $\CSP \leq \Cphi\,(1+\CP)$ according to
Proposition~\ref{prop:poincares}.\qed

%%%%%%%%%%%%%%%%%%%%%%%%%%%%%%%%%%%%%%%%%%%%%%%%%%%%%%%%%%%%%%%
%%%%%%%%%%%%%%%%%%%%%%%%%%%%%%%%%%%%%%%%%%%%%%%%%%%%%%%%%%%%%%%
\section{Operators on vector fields: proof of
  Theorem~\texorpdfstring{\ref{theo:ao}}{Theorem3}}
\label{sec:ao}

In this section we develop the functional analysis and the
spectral theory of operators on vector fields, and prove
Theorem~\ref{theo:ao}. All results on the tensorized operator
$-\Delta_\phi$ on vector fields are direct consequences of the
study of the corresponding scalar operator: from
Section~\ref{sec:wittenpoincare}, we learn that $-\Delta_\phi$ is
essentially self-adjoint and admits $\ccc_c^\infty(\R^d;\R^d)$ as
a core, the domain of its unique self-adjoint extension is
\begin{equation*}
  D(-\Delta_\phi)=\Big\{u \in L^2\,:\,\forall\,j \in \set{
    1,\cdots, d}, \;\|\wgt^2\,u_j\|^2+\big\|\wgt\,\nabla
  u_j\big\|^2+\big\|D^2 u_j\big\|^2<\infty \Big\}\,,
\end{equation*}
and its kernel is $\ker(-\Delta_\phi)=\R^d$.

Let us deal with the other operators of
Theorem~\ref{theo:ao}. Recall that the operator $-\Delta_S$ is
defined on $\ccc_c^\infty(\R^d;\R^d)$ vector fields by
$ -\Delta_S=-\,D^s_\phi\cdot D^s$. It is nonnegative and
$\Id -\Delta_S$ has therefore a Friedrichs extension with domain
included in $H^1_S$ defined as the completion of
$\ccc_c^\infty(\R^d;\R^d)$ with respect to the norm given by
$u\mapsto\|u\|^2+\|D^s u\|^2$. On the other hand, a maximal
self-adjoint extension of $\Id -\Delta_S$ can be built according
to the Lax-Milgram Theorem and its domain is included in
$H^1$. The Korn inequality~\eqref{eq:WKfull} implies that
$H^1_S=H^1$ so that the two extensions coincide, since there is a
unique extension for which the domain is contained in $H^1_S$
(\cite[Theorem~X.23]{RS75}), which is the case for the maximal
one. We have proven that $-\Delta_S$ is essentially
self-adjoint. From the Poincar\'e-Korn
inequality~\eqref{eq:WPKfull}, we learn that
$\ker (-\Delta_S)=\RR\oplus\R^d$ and that
$\inf\big(\Spec(-\Delta_S) \cap(0,+\infty)\big) \geq \CPK^{-1}
>0$. This concludes the proof of Theorem~\ref{theo:ao} for
$-\Delta_S$.

The same argument applies to
$-\Delta_{S\phi}=-\Delta_{S}-\nabla\phi \otimes \nabla\phi $
using~\eqref{eq:WKfull}--\eqref{eq:WPKstrong} as we know from
Proposition~\ref{prop:domain} that
$ \nabla\phi \otimes \nabla\phi\,u \in L^2$ for all $u \in
D$. This completes the proof of of Theorem~\ref{theo:ao}.
%--------------------------------------------------------------
\begin{rem}
  Note also that the alternative operator defined on smooth
  vector fields by
  $u\mapsto-\,D^s_\phi\cdot D^s u-\nabla(\nablaphi\cdot u)$ has
  exactly the same properties as $-\Delta_S$ because
  $\|D^s u\|^2+\|\nabla\phi\cdot u\|^2 \sim\|D^s
  u\|^2+\|\nablaphi\cdot u\|^2$ where
  $\nablaphi u :=\nabla\cdot u- \nabla\phi\cdot u$ and
  $\nabla\cdot u=\mathrm{Tr}(D^su)$.
\end{rem}
%--------------------------------------------------------------

%%%%%%%%%%%%%%%%%%%%%%%%%%%%%%%%%%%%%%%%%%%%%%%%%%%%%%%%%%%%%%%
%%%%%%%%%%%%%%%%%%%%%%%%%%%%%%%%%%%%%%%%%%%%%%%%%%%%%%%%%%%%%%%
\section{Zeroth order Korn inequalities: proof of
  Theorem~\ref{theo:KPK0}}
\label{sec:kornzero}

In order to prove~\eqref{eq:WKZfull}, we use a new
Poincar\'e-Lions-type inequality of order $-1$ and the Schwarz
Lemma.

%%%%%%%%%%%%%%%%%%%%%%%%%%%%%%%%%%%%%%%%%%%%%%%%%%%%%%%%%%%%%%%
\subsection{A Poincar\'e-Lions inequality of order
  \texorpdfstring{$-1$}{-1}}
\label{Sec:6.1}
%--------------------------------------------------------------
\begin{lem}
  \label{LemPL-1}
  There exists two positive constants $\CLPL$ and $\CRPL$ such
  that, for all $f \in H^{-1}$, we have
  \begin{equation}
    \label{eq:poincarelions-1}
    \CLPL^{-1}\,\|\Lambda^{-1/2}(f-\seq{f})\|^2
    \leq\|\Lambda^{-1}\,\nabla f\|^2 \leq
    \CRPL\,\|\Lambda^{-1/2}\,(f-\seq{f})\|^2.
  \end{equation}
\end{lem}
%--------------------------------------------------------------
\begin{proof} We rely on the same strategy as for the proof of
  the Poincar\'e-Lions inequality~\eqref{eq:poincarelions}. For
  any $f \in H^{-1}$, the mean makes sense because
  $\seq{f}=\Lambda^{-1/2}\,\seq{f}=\seq{\Lambda^{-1/2}\,f}$ as
  $\Lambda=\Id$ when restricted on constants. We can therefore
  take $\seq{f}=0$ w.l.o.g.~and apply the spectral theorem as
  in~\eqref{eq:spectralthm}, for any $f \in D(\Lambda)$, to get
  \begin{equation}
    \label{eq:spectralthmbis}
    (1+\CP)^{-1}\,\|\Lambda^{-1/2}\,f\|^2 \leq
    \sep{(-\Delta_\phi)\,\Lambda^{-1}
      \Lambda^{-1/2}\,f,\Lambda^{-1/2}\,f}=
    \sep{\Lambda\nabla\Lambda^{-3/2}
      (\Lambda^{-1/2}\,f), \Lambda^{-1}\nabla f},
  \end{equation}
  where we used that $-\Delta_\phi=-\,\nablaphi\cdot \nabla$ and
  $\Lambda$ commute. In order to prove the left inequality
  in~\eqref{eq:poincarelions-1}, it is sufficient to prove that
  $\Lambda\nabla\Lambda^{-3/2}$ is a bounded operator. We work in
  $\ccc_c^\infty(\R^d;\R)$, which is a core for $\Lambda$, and
  the conclusion follows by density in $L^2$. Let us write
  \begin{equation*}
    \Lambda\nabla\Lambda^{-3/2}=\nabla\Lambda^{-1/2}+[\Lambda,\nabla]
    \, \Lambda^{-3/2}=\nabla\Lambda^{-1/2}-D^2\phi\,\nabla
    \Lambda^{-3/2} = \nabla\Lambda^{-1/2}-D^2\phi\,\Lambda^{-1}\,
    \big(\Lambda \nabla\Lambda^{-3/2}\big)\,.
  \end{equation*}
  By assumption~\eqref{hyp:regularity}, for all $\eps>0$ and for
  all $g \in \ccc^{\infty}(\R^d;\R)$, we know that
  \begin{align*}
    \|\Lambda \nabla\Lambda^{-3/2} g\|
    & \leq\|\nabla\Lambda^{-1/2}\,g\|
      +\eps\,\|\wgt^2\,\nabla\Lambda^{-3/2}g\|
      +C_\eps\,\|\nabla\Lambda^{-3/2}g\|\\
    & \leq\|\nabla\Lambda^{-1/2}\,g\|
      +\eps\,\big\|\wgt^2\,\Lambda^{-1}
      \big(\Lambda \nabla\Lambda^{-3/2}g\big)\big\|
      +C_\eps\,\|\nabla\Lambda^{-1/2}\,(\Lambda^{-1}g)\|\,.
  \end{align*}
  The operators $\nabla\Lambda^{-1/2}$ and $\wgt^2\,\Lambda^{-1}$
  are bounded respectively by $1$ and $\sqrt{\CB}$ according
  to~\eqref{eq:toolboxH1} and~\eqref{eq:toolbox}, and
  $\Lambda^{-1}\le1$, so that
  \begin{equation*}
    \|\Lambda \nabla\Lambda^{-3/2} g\|
    \leq (1+C_\eps)\,\|g\|+\eps\, \sqrt{\CB}\,
    \|\Lambda \nabla\Lambda^{-3/2}g\|\,.
  \end{equation*}
  With the choice $\eps=1/(2\,\sqrt{\CB})$ and
  $\Cphi'':=C_\eps=C_{1/(2\,\sqrt{\CB})}$
  in~\eqref{hyp:regularity}, we obtain
  \begin{equation*}
    \|\Lambda \nabla\Lambda^{-3/2} g\|\leq 2\,(1+\Cphi'')\,\|g\|\,.
  \end{equation*}
  Coming back to~\eqref{eq:spectralthmbis} with
  $g=\Lambda^{-1/2}\,f$, we obtain
  \begin{equation*}
    \|\Lambda^{-1/2}\,f\|\leq 2\,(1+\CP)\,(1+\Cphi'')\,
    \|\Lambda^{-1}\,\nabla f\|\,,
  \end{equation*}
  so that $\CLPL \leq 4\,(1+\CP)^2\,(1+\Cphi'')^2$ and the left
  inequality is proven.

  In order to prove the right inequality
  in~\eqref{eq:poincarelions-1}, we notice that
  $\Lambda^{-1}\,\nabla f
  =\Lambda^{-1}\,\nabla\Lambda^{1/2}\,(\Lambda^{-1/2}\,f)$ and it
  is therefore sufficient to prove that
  $\Lambda^{-1}\nabla\Lambda^{1/2}$ is a bounded operator. This
  is done as in~\eqref{eq:estimatecpl1new} by writing
  \begin{equation*}
    \Lambda^{-1}\nabla\Lambda^{1/2}=
    \Lambda^{-1}\nabla\Lambda\,\Lambda^{-1/2}=
    \nabla\Lambda^{-1/2}+\Lambda^{-1}
    [\nabla,\Lambda]\,\Lambda^{-1/2}=
    \nabla\Lambda^{-1/2}+\Lambda^{-1}\,D^2\phi\,
    \big(\nabla\,\Lambda^{-1/2}\big)\,.
  \end{equation*}
  As in the proof of~\eqref{eq:estimatecpl1new}, we obtain for
  any $f\in\ccc_c^\infty(\R^d;\R)$ and $g=\Lambda^{-1/2}\,f$ the
  estimate
  \begin{equation*}
    \|\Lambda^{-1}\nabla f\|=\|\Lambda^{-1}\nabla\Lambda^{1/2}g\|
    \leq\(1+\tfrac14\,\Cphi'\,{\textstyle\sqrt{\CB/C_\phi}}\,\)
    \|\nabla\,\Lambda^{-1/2}g\|\leq
    \(1+\tfrac14\,\Cphi'\,{\textstyle\sqrt{\CB/C_\phi}}\,\)
    \|g\|\,,
  \end{equation*}
  using~\eqref{eq:cphi},~\eqref{eq:toolboxH1}
  and~\eqref{eq:toolboxadj}. This concludes the proof with
  $\CRPL=\big(1+\tfrac14\,\Cphi'\,\sqrt{\CB/C_\phi}\,\big)^2$.
\end{proof}

%%%%%%%%%%%%%%%%%%%%%%%%%%%%%%%%%%%%%%%%%%%%%%%%%%%%%%%%%%%%%%%%%%%%%%
\subsection{Proof of the Korn inequalities in Theorem~\ref{theo:KPK0}}

As a consequence of~\eqref{eq:fpbounded} the projection $u \mapsto \fP(Du)=\seq{\nablaskew u}$ has a unique extension as a bounded operator on $L^2$ with norm bounded by $\|\nabla{\phi}\|$ since~$H^1$ is dense in $L^2$. We keep the same name for the
extension and notice that
\begin{equation*}
\Lambda^{-1/2}\,\fP(Du)=\Lambda^{-1/2}\,\seq{\nablaskew u}=\langle \Lambda^{-1/2}\,\nablaskew u \rangle=
\fP( \Lambda^{-1/2}\,Du)\,.
\end{equation*}

\noindent$\rhd$ {\em Proof of~\eqref{eq:WKZfull}}. Let us take $u\in L^2$ such that $\seq{u}=0$ and $\fP(Du)=\seq{\nablaskew u}=0$. Using~\eqref{eq:poincarelions-1}, we have
\begin{equation*}\label{eq:WKZfull1}
\|\Lambda^{-1/2}\,Du\|^2=\|\Lambda^{-1/2}\,\nablasym u\|^2+\|\Lambda^{-1/2}\,\nablaskew u\|^2\leq\|\Lambda^{-1/2}\,\nablasym u\|^2+\CLPL\,\|\Lambda^{-1}\,\nabla(\nablaskew u)\|^2
\end{equation*}
where $\|\Lambda^{-1}\,\nabla(\nablaskew u)\|^2=\sum_{i,j=1}^d\|\Lambda^{-1}\,\nabla (\nablaskew u)_{ij}\|^2$. By the Schwarz Theorem~\eqref{Schwarz},
\begin{equation*}\label{eq:WKZfull3}
\|\Lambda^{-1}\,\nabla(\nablaskew u)\|^2\leq 2 \sum_{i,j,k=1}^d\Big(\|\Lambda^{-1} \D_i (\nablasym u)_{jk}\|^2+\|\Lambda^{-1} \D_j (\nablasym u)_{ik}\|^2\Big)=4 \sum_{j,k=1}^d\|\Lambda^{-1}\,\nabla (\nablasym u)_{jk}\|^2,
\end{equation*}
and~\eqref{eq:poincarelions-1} yields
\begin{equation*}
  \sum_{i,j=1}^d\|\Lambda^{-1}\,\nabla (\nablasym u)_{ij}\|^2
  \leq \CRPL\sum_{j,k=1}^d\|\Lambda^{-1/2}\,(\nablasym
  u)_{jk}\|^2 =\CRPL\,\|\Lambda^{-1/2}\,\nablasym u\|^2.
\end{equation*}
Altogether, this proves
\begin{equation*}
  \label{eq:WKZfullbis}
  \|\Lambda^{-1/2}\,\big(Du-\fP(Du)\big)\|^2 \leq (1+4\,\CLPL\, 
  \CRPL )\,\|\Lambda^{-1/2}\,\nablasym u\|^2
\end{equation*}
and~\eqref{eq:WKZfull} follows with $\CKZ=1+4\,\CLPL\,\CRPL$.
\smallskip

\noindent$\rhd$ {\em Proof of~\eqref{eq:WPKZfull}}. Let us take
$u\in L^2$ such that $\seq{u}=0$ and $\P(u)=0$. By definition of
$\P$,~\eqref{eq:poincarelions} and~\eqref{eq:WKZfull} we get
\begin{equation*}
  \norm{u}^2 \leq \norm{u- \fP(Du)\,x} \leq
  \CPL\,\norm{\Lambda^{-1/2}\,(Du- \fP(Du))}^2 \leq
  \CPL\,(1+4\,\CLPL\,\CRPL)\,\|\Lambda^{-1/2}\,\nablasym u\|^2.
\end{equation*}
This proves~\eqref{eq:WPKZfull} with
$\CPKZ \leq \CPL\,(1+4\,\CLPL\,\CRPL)$.
\smallskip

\noindent$\rhd$ {\em Proof of~\eqref{eq:WPKzero}}. Let us
consider $u\in L^2$ such that $\P_\phi(u)=0$ so that
$\P(u) \in \RR_\phi^c$. By~\eqref{eq:WPKZfull} and by
definition~\eqref{eq:rigidityzero}, we have
\begin{equation*}
\begin{split}
  \norm{u}^2
  &=\norm{u-\P(u)-\seq{u}}^2+\norm{\P(u)+\seq{u}}^2\\
  &\leq \CPKZ\,\|\Lambda^{-1/2}\,\nablasym
  u\|^2+\CRVZ\,
  \|\Lambda^{-1/2}\,\nabla\phi\cdot (\P(u)+\seq{u})\|^2\\
  &\leq \CPKZ\,\|\Lambda^{-1/2}\,\nablasym
  u\|^2+2\,\CRVZ\,\|\Lambda^{-1/2}\,(\nabla\phi\cdot
  u)\|^2+2\,\CRVZ\,\|\Lambda^{-1/2}\,[\nabla\phi\cdot
  (u-\P(u)-\seq{u})]\|^2.
\end{split}
\end{equation*}
Inequalities~\eqref{eq:toolboxH1adj}, $\Lambda^{-1/2}\le1$
and~\eqref{eq:WPKZfull} yield
\begin{equation*}
\begin{split}
  \norm{u}^2 & \leq \CPKZ\,\|\Lambda^{-1/2} \nablasym
  u\|^2+2\,\CRVZ\,\|\Lambda^{-1/2} (\nabla\phi\cdot
  u)\|^2+2\,\CRVZ\,\Cphi\,\|u-\P(u)-\seq{u}\|^2\\
  & \leq \CPKZ\|\Lambda^{-1/2} \nablasym u\|^2+2\,\CRVZ
  \norm{\Lambda^{-1/2} (\nabla\phi\cdot
    u)}^2+2\,\CRVZ\,\Cphi\,\CPKZ\,\|\Lambda^{-1/2} \nablasym u\|^2.
\end{split}
\end{equation*}
This proves~\eqref{eq:WPKzero} with
$\CPKZ' \leq \CPKZ(1+2\,\CRVZ\,\Cphi)$, and also completes the
proof of Theorem~\ref{theo:KPK0}.\qed

%%%%%%%%%%%%%%%%%%%%%%%%%%%%%%%%%%%%%%%%%%%%%%%%%%%%%%%%%%%%%%%%%%
%%%%%%%%%%%%%%%%%%%%%%%%%%%%%%%%%%%%%%%%%%%%%%%%%%%%%%%%%%%%%%%%%%
\appendix\section{Extensions, geometric observations and
  motivation from kinetic theory}
\label{Appendix:A}

%%%%%%%%%%%%%%%%%%%%%%%%%%%%%%%%%%%%%%%%%%%%%%%%%%%%%%%%%%%%%%%%%%
\subsection{On the assumptions and some generalizations}
\label{Sec:Remarks}

\begin{rem}
  \label{rem:concentration}
  The fact that the Poincar\'e
  inequality~\eqref{hyp:poincarebasique} implies that
  $e^{-\phi} \dd x$ has an average $\seq{x}$ and a variance
  $\langle |x|^2 \rangle$ is classical (see, \eg,~\cite[Corollary
  3.2]{Led01}). Indeed~\eqref{hyp:poincarebasique} yields a
  concentration property of $e^{-\phi(x)} \dd x$ via a {\em
    concentration function}. A direct application of Fatou's
  Lemma allows to extend~\eqref{hyp:poincarebasique} to the set
  $W^{1,\infty}$ of uniformly Lipschitz functions, which includes
  $x\mapsto x_j$ for all $j\in \set{ 1,\cdots , d}$. This
  directly gives
  $\int_{\R^d} |x|^2\,e^{-\phi(x)} \dd x \leq d \,\CP$ and the
  integrability of $x$ w.r.t.~$e^{-\phi(x)} \dd x$ follows by the
  Cauchy-Schwarz inequality. By induction, we get
  under~\eqref{hyp:poincarebasique} alone that the set of all
  polynomial functions $\R[x]$ is included in $L^2$ and even
  $H^1$.
\end{rem}

\begin{rem}
  \label{rem:integrability}
  As a consequence of Remark~\ref{rem:concentration} and of the
  strong Poincar\'e inequality~\eqref{eq:strongpoincare}, we
  directly get that for all $i,j \in \set{1,\cdots d}$,
  $\int_{\R^d} |x_i\,\D_j \phi(x)|\,e^{-\phi(x)} \dd x
  <\infty$. It is indeed sufficient to apply the strong
  Poincar\'e inequality to $x \mapsto x_j$, which is in
  $H^1$. Note that this gives sense to all quantities of
  Theorem~\ref{theo:KPK}, \eg, $\norm{\nabla\phi\cdot R}$ for any
  infinitesimal rotation $R$.
\end{rem}

\begin{rem}
  \label{rem:poincare}
  There are many sufficient conditions for the Poincar\'e
  inequality of
  Assumption~\eqref{hyp:poincarebasique}. When~$\phi$ is
  uniformly convex, it is shown in~\cite{Bakry-Emery} that
  $C_{\mathrm P}$ is greater or equal than the convexity constant,
  hence leading to fully explicit estimates in the two main
  theorems. If we only assume that
  $\lim_{|x| \mapsto\infty } |\nabla\phi(x)|=+\infty$,
  then~$\Lambda$ is in fact an operator with compact resolvent,
  hence with discrete spectrum, and~\eqref{hyp:poincarebasique}
  follows. Another, less stringent, sufficient condition on
  $\phi$ is
  $\liminf_{|x| \to \infty} \big(\frac12
  |\nabla\phi(x)|^2-\Delta\phi(x)\big) >c$ for some $c>0$ (it
  implies the Poincar\'e inequality from the Persson-Agmon formula
  of~\cite{Per60} or~\cite[Theorem 3.2]{Agm82}). We note that this
  last assumption is satisfied by any regular function which
  coincides with $x \mapsto \alpha\,|x|+\beta$ outside of a large
  centred ball, where $\alpha$ and $\beta$ are normalization
  constants.
\end{rem}

\begin{rem}\label{rem:boundedcase}
  Assumptions~\eqref{hyp:intnorm}--\eqref{hyp:regularity}--\eqref{hyp:poincarebasique}
  may be satisfied in other geometries than the one of the whole
  Euclidean space $\R^d$. In particular, given an open, smooth,
  bounded and connected subset $\Omega$ of $\R^d$, we observe
  that these hypotheses are satisfied by the potential
  $\phi(x)=\exp\big(1/d^2(x,\partial\Omega)\big)$, where
  $d(x,\partial\Omega)$ denotes the usual Euclidean distance from
  $x$ to $\partial\Omega$. Here $\R^d$ is replaced by $\Omega$
  equipped with the measure $e^{-\phi(x)} \dd x$. It is an open
  question to understand how our results could be extended to
  usual boundary problems with potentials mimicking walls at the
  boundary of $\Omega$.
\end{rem}

%%%%%%%%%%%%%%%%%%%%%%%%%%%%%%%%%%%%%%%%%%%%%%%%%%%%%%%%%%%%%%
\subsection{Rigidity constants and defects of axisymmetry}
\label{subsec:rigidity}

In the proofs of Theorems~\ref{theo:KPK} and~\ref{theo:KPK0} (see
also Appendix~\ref{Sec:Constants}), we used the two rigidity
constants $\CRV$ and $\CRVZ$ defined by~\eqref{eq:rigidityvect}
and~\eqref{eq:rigidityzero}, depending on the level of regularity
in each case, to measure the defects of axisymmetry (note that
$\CRV \leq \CRVZ $ since $\Lambda \geq \Id$). We used also $\CRD$
defined in~\eqref{eq:rigiditydiff} but remark that $\CRV$ and
$\CRD$ are not directly comparable either, because
$\fM_\phi^c \neq D \RR_\phi^c$. Other ways of measuring the
default of axisymmetry of the potential $\phi$ can be considered.
\smallskip

\noindent
\circled1 One can consider, again, a {\em rigidity of vector
  fields} constant, but this time defined alternatively by
\begin{equation*}
  \CRVL^{-1}=\min_{ A\,x \in \RR_\phi^c \setminus \set{0}}
  \frac{\norm{\nabla\phi\cdot A\,x}^2}{\norm{A\,x}^2}\quad
  \mbox{when}\quad \RR_\phi^c \neq \set{0}\quad\mbox{and}
  \quad\CRVL=0\quad\mbox{otherwise}\,.
\end{equation*}
This leads to the {\em modified Poincar\'e-Korn inequality}
\begin{equation}
  \label{eq:WPKL}
  \inf_{A\,x\in\RR_\phi}\|u- \seq{u}-A\,x\|^2=\|u- \seq{u}-
  \P_\phi (u)\|^2 \leq \CPKL'\,\|D^s u\|^2
  +2\,\CRVL\,\|\nabla \phi\cdot (u-\seq{u})\|^2
\end{equation}
with an explicit bound for the constant $\CPKL'$
using~\eqref{eq:WPKfull} and the method of proof
of~\eqref{eq:WPK}. Once more, the existence of $\CRVL$ follows
from the injectivity of $A\,x \mapsto \nabla\phi\cdot A\,x$ on
$\RR_\phi^c$ and the fact that $\RR_\phi^c$ is of finite
dimension. The main advantage of this approach is to preserve a
continuity property with respect to axisymmetry, which can be
stated as follows: a small perturbation of a radial potential
$\phi$ gives rise to a small constant $\CRVL$, the limiting case
being $\RR_\phi^c=\set{0}$ and $\CRVL=0$. The main drawback is
that the symmetric operator associated to~\eqref{eq:WPKL} is
neither local nor differential because of the term $\seq{u}$
which appears in the right-hand side of~\eqref{eq:WPKL}.
\smallskip

\noindent
\circled2 In a bounded domain $\Omega\subset\R^d$, with flat
metric ({\em i.e.}, for a constant potential $\phi$) considered
in~\cite{DV02}, the authors use {\em Grad's number}. Let us
explain how to adapt this method in our context under, \eg, the
additional condition
\begin{equation*}
  \lim_{|x| \rightarrow
    \infty}\frac{D^2\phi(x)}{\lfloor\nabla\phi(x)\rceil^2}=0\,.
\end{equation*}
This property implies that the multiplication operator by
$D^2 \phi$ is relatively compact with respect to $-\Delta_\phi$
acting on vector fields, with essential spectrum in
$[\CP, \infty)$. The spectrum in $[0,\CP)$ is then a pure point
spectrum and the kernel is finite dimensional. For any
antisymmetric matrix $\antiSigma$, there exists an affine space
$\VA$ of functions $v\in H^1$ solving {\em the Witten-Hodge
  problem}
\begin{equation*}\label{eq:uA}
  \Divphi v=0\,,\quad \nablaskew v=\antiSigma\,.
\end{equation*}
The {\em Witten-Hodge inequality} asserts that
\begin{equation}\label{eq:IntroWHineq}
  \inf_{v \in \VA}\|D^s v\|^2 \le c_{\mathrm H}\,|\antiSigma|^2
\end{equation}
for some constant $c_{\mathrm H}\in (0,\infty)$. The reverse
inequality amounts to the existence of {\em Grad's number} such
that
\begin{equation*}\label{gradnumber}
C_{\mathrm G}^{-1} :=\inf_{\antiSigma\in \fM_\phi^c,\,|\antiSigma|=1,\,v \in \VA}\;\|\nablasym v\|^2.
\end{equation*}
The existence of $C_{\mathrm G}$ as well as a quantitative
positive lower bound could be establish using mass transport
theory exactly as in~\cite{DV02}. Of course, $C_{\mathrm G}$ is
well defined only when $\RR_\phi^c\neq\{0\}$, \ie, under the
condition that $\phi$ is not radially
symmetric. Inequality~\eqref{eq:IntroWHineq} is natural in
differential geometry and more specifically in De Rham cohomology
theory: we refer to~\cite{AS00,HN05} for further developments on
this topic. In bounded domains, how to measure the symmetry
defect by Grad's number in view of Korn type inequalities is at
the core of~\cite{DV02} but has also been studied
in~\cite{MR2542573}. This approach differs from ours. Using
$C_{\mathrm G}$ and~\eqref{eq:WKfull}, the
inequality~\eqref{eq:WPK} can be proved along a similar 
strategy as in~\cite{DV02}, although with different constants.

%%%%%%%%%%%%%%%%%%%%%%%%%%%%%%%%%%%%%%%%%%%%%%%%%%%%%%%%%%%%%%%%%%%%%%
\subsection{An elementary application in kinetic theory}
\label{subsec:toymodel}

The main motivation for this paper comes from {\em kinetic
  equations involving a confining potential} studied
in~\cite{CDHMMShypo}. Also see~\cite[Section~2]{DV02}
and~\cite{MR3815207,Duan_2011} for applications of Korn
inequalities to kinetic equations. As an example, let us consider
the linear relaxation model of {\em BGK}-type
\begin{equation}\label{eq:toykinintro}
  \partial_tf+v\cdot\nabla_x f -\nabla_x \phi\cdot \nabla_v f
  =\mathcal Lf :=G_f-f,
\end{equation}
where $f(t,x,v)$ is an unknown distribution function for a system
of particles depending on time $t \ge 0$, position $x\in\R^d$ and
velocity $v\in\R^d$, and where $G_f$ is defined by
\begin{equation*}
  G_f:=(\rho+\bu\cdot v)\,\mu\quad\mbox{where}\quad \rho(t,x)
  :=\int_{\R^d} f(t,x,v)\,\dd v\,,\quad \bu(t,x)
  :=\int_{\R^d}v\,f(t,x,v)\,\dd v\,.
\end{equation*}
Here $\mu(v):=(2\,\pi)^{-d/2}\,e^{-|v|^2/2}$ while $\rho $ and
$\bu$ are respectively the macroscopic density and the average
velocity associated with $f$. The collision kernel admits $d+1$
conserved moments, in the sense that
$\int_{\R^d}\mathcal Lf(t,x,v)\,\dd v=0=\int_{\R^d}v_i\,\mathcal
Lf(t,x,v)\,\dd v$ for any $i=1,\ldots,d$ and $f \in
L^1((1+|v|){\rm d} v)$.

A natural question is to look for equilibria
of~\eqref{eq:toykinintro}. A quick glance at the equation shows
that $\MM(x,v) :=e^{-\phi(x)}\,\mu(v)$ is one of them. Korn
inequalities provide us with a complete answer.
%------------------------------------------------------------
\begin{proposition}
  Under
  Assumptions~\eqref{hyp:intnorm}--\eqref{hyp:regularity}--\eqref{hyp:poincarebasique},
  all equilibria of~\eqref{eq:toykinintro} in
  $L^2(\mmm^{-1}\dd x\dd v)$ take the form
  $f(x,v)=\big(( R(x)\cdot v)+c\big)\,\MM$ for some
  $R \in \RR_\phi$ and $c \in \R$.
\end{proposition}
%---------------------------------------------------------------------
\begin{proof} 
  Write $f=h\,\mmm$ with $h \in L^2(\mmm \dd x\dd v)$,
  $\rho=r\,e^{-\phi}$, $\bu=u\,e^{-\phi}$, so that
  equation~\eqref{eq:toykinintro} reads
  \begin{equation}
    \label{eq:toykinintrobis}
    \partial_th+v\cdot\nabla_x h -\nabla_x \phi\cdot \nabla_v h =
    L(h) :=h-r-u\cdot v.
  \end{equation}
  The restriction of $L$ to $L^2(\mu \dd v)$ is $L=-\,\Pi^\bot$
  where $\Pi$ is the orthogonal projection onto
  $\mathrm{Span} \{1, v_1, \ldots, v_d\}$. We compute
  \begin{equation*}
    \frac d{dt}\iint_{\R^d\times\R^d}|h|^2\,\MM \dd x\dd v = 
    2\iint_{\R^d\times\R^d}(L h)\,h\,\MM \dd x\dd v=
    -\,2 \iint_{\R^d \times \R^d} \left| \Pi^\bot h \right|^2\, 
    \MM \dd x \dd v
  \end{equation*}
  and deduce that any stationary solution
  of~\eqref{eq:toykinintrobis} takes the form
  $h(x,v)=r(x)+u(x)\cdot v $. Equation~\eqref{eq:toykinintrobis}
  then reads
  \begin{equation*}
    v\cdot\nabla_x(r+u\cdot v)=\nabla_x\phi\cdot u.
  \end{equation*}
  Integrating the latter equation against respectively $1$, $v_i$
  and $v_i\,v_j$ with $i \not=j$ in $L^2(\mu \dd v)$ yields
  \begin{equation*}
    i)\quad\nabla_x\cdot u-\nabla_x \phi\cdot u=0,\quad 
    ii)\quad D^s u=0,\quad iii)\quad\nabla r=0.
  \end{equation*}
  From iii) we get that there exists $c \in \R$ such that
  $r=c$. As for i), an integration by parts gives
  \begin{equation*}
    0=\int_{\R^d} (\nabla_x\cdot u-\nabla_x \phi\cdot u)
    \seq{u}\cdot x\,e^{-\phi(x)} \dd x=
    -\,\int_{\R^d} u\cdot \nabla
    \sep{ \seq{u}\cdot x}\,e^{-\phi(x)} \dd x=
    -\,\int_{\R^d} u\cdot \seq{u}\,e^{-\phi(x)} \dd x=-\,\seq{u}^2,
  \end{equation*}
  so that $\seq{u}=0$. Note also that taking the trace in ii)
  yields $\nabla_x\cdot u=0$ so that i) reads
  $\nabla_x \phi\cdot u=0$. Using this and~\eqref{eq:WPKZfull} in
  Theorem~\ref{theo:KPK0} shows that $u=R$ with
  $R=\P_\phi(u)\in\RR_\phi$. Hence $h=R(x)\cdot v+c$ and
  $f(x,v)=\big(( R(x)\cdot v)+c\big)\,\MM$. The reciprocal is
  straightforward, which completes the proof.
\end{proof}

%%%%%%%%%%%%%%%%%%%%%%%%%%%%%%%%%%%%%%%%%%%%%%%%%%%%%%%%%%%%%%%%
%%%%%%%%%%%%%%%%%%%%%%%%%%%%%%%%%%%%%%%%%%%%%%%%%%%%%%%%%%%%%%%%
\section{Additional details on computations}
\label{Appendix:B}

%%%%%%%%%%%%%%%%%%%%%%%%%%%%%%%%%%%%%%%%%%%%%%%%%%%%%%%%%%%%%%%
\subsection{Functions, derivatives and projections}
\label{Appendix:B2}

We denote by $f$ a generic scalar function on
$\R^d$ and by $u:\R^d\to\R^d$ a generic vector field, so that
\hbox{$\nabla f=(\partial_if)_{i=1}^d$} is a vector field and
$Du=(\partial_ju_i)_{i,j=1}^d$ takes values in $\mathfrak M$. The
symmetric and the antisymmetric differentials of $u$,
respectively $D^su=\big((D^su)_{ij}\big)_{i,j=1}^d$ and
$D^au=\big((D^au)_{ij}\big)_{i,j=1}^d$ are defined by
\begin{equation*}
(D^su)_{ij}:=\tfrac12\(\partial_ju_i+\partial_iu_j\)\quad\mbox{and}\quad (D^au)_{ij}:=\tfrac12\(\partial_ju_i-\partial_iu_j\)
\end{equation*}
so that $D^su+D^au=Du$.

The orthogonal projection $\P$ of vector-valued functions is
defined as follows. Let $(A_{ij})_{1\le i<j\le d}$ be a basis of
$\mathfrak M^a$ whose elements are
\begin{equation*}
  A_{ij}=\big((\delta_{ij}-\delta_{ji})\,\delta_{ki}\,\delta_{j\ell}
  \big)_{k,\ell=1}^d
\end{equation*}
and $(R_{ij})_{1\le i<j\le d}$ the orthonormal (in the $L^2$
sense) basis of $\mathcal R$ given by
\begin{equation*}
  R_{ij}(x)=Z_{ij}^{-1}\,A_{ij}\,x
\end{equation*}
whose coordinates are all $0$ except the $i^{th}$ and the
$j^{th}$ ones, with respective values $-x_j/Z_{ij}$ and
$x_i/Z_{ij}$, {\em i.e.},
\begin{equation*}
  R_{ij}(x)^\perp=Z_{ij}^{-1}\,\big(0,\ldots,0,-x_j,0,
  \ldots,0,x_i,0,\ldots,0\big)\,,
\end{equation*}
and where the normalization constant is
$Z_{ij}=\(\int_{\R^d}(x_i^2+x_j^2)\,e^{-\phi(x)}\dd
x\)^{1/2}$. With these notations, $\P u$ is the vector field
\begin{equation*}
  x\mapsto\P u(x):=\sum_{1\le i<j\le d}\mathsf c_{ij}\,R_{ij}(x)
\end{equation*}
where the coefficients are computed, for all integers $i$, $j$
such that $1\le i<j\le d$, as
\begin{equation*}
  \textstyle\mathsf c_{ij}=
  \int_{\R^d}u(x)\cdot R_{ij}(x)\,e^{-\phi(x)}\dd x=
  \frac1{Z_{ij}}\int_{\R^d}\(x_i\,u_j(x)-x_j\,u_i(x)\)\,
  e^{-\phi(x)}\dd x\,.
\end{equation*}

The orthogonal projection $\fP$ of a matrix-valued function
$\mathfrak F$ is defined as
\begin{equation*}
  \fP\,\mathfrak F:=\sum_{1\le i<j\le d}\mathsf d_{ij}\,A_{ij}
\end{equation*}
where the coefficients are computed, for all integers $i$, $j$
such that $1\le i<j\le d$, as
\begin{equation*}
  \textstyle\mathsf d_{ij}=
  \frac12\int_{\R^d}\mathfrak F(x):A_{ij}\,e^{-\phi(x)}\dd x=
  \frac12\int_{\R^d}
  \big(\mathfrak F_{ij}(x)-\mathfrak F_{ji}(x)\big)\,
  e^{-\phi(x)}\dd x\,.
\end{equation*}
As a a consequence, we deduce that
$\fP\mathfrak F=\seq{\mathfrak F^a}$ and $\fP(Du) = \seq{D^a u}$
for any $u \in H^1$.

A matrix $A\in D\RR_\phi^c$ is such that $A\in\fM^a$ and for any
$B\in D\RR_\phi\subset\fM^a$,
\begin{equation*}
  \textstyle0=\int_{\R^d}A\,x\cdot B\,x\,e^{-\phi(x)}\dd x=
  \sum_{i,j,k=1}^dA_{ij}\,B_{ik}
  \int_{\R^d}x_j\,x_k\,e^{-\phi(x)}\dd x\,.
\end{equation*}
A matrix $A\in\fM_\phi^c$ is such that $A\in\fM^a$ and for any
$B\in D\RR_\phi=\fM_\phi\subset\fM^a$,
\begin{equation*}
  \textstyle0=\int_{\R^d}A:B\,e^{-\phi(x)}\dd x=
  A:B=\sum_{i,j}^dA_{ij}\,B_{ij}.
\end{equation*}
Based on these two definitions, it is clear that $D\RR_\phi^c$
and $\fM_\phi^c$ generically differ.

%%%%%%%%%%%%%%%%%%%%%%%%%%%%%%%%%%%%%%%%%%%%%%%%%%%%%%%%%%%%%%
\subsection{Operators}
\label{Appendix:B4}

Let us give some details on the differential operators
$-\Delta_\phi$ and $-\Delta_S$ associated respectively with the
quadratic forms $f\mapsto\|\nabla f\|^2$ and
$u\mapsto\|D^s u\|^2$.
\smallskip

\noindent$\rhd$ Using
$\nablaphi u := \nabla\cdot u- \nabla \phi\cdot u$, the {\sl
  Witten-Laplace operator $\Delta_\phi$ on functions} is such
that $\|\nabla f\|^2=(f,-\Delta_\phi f)$ and takes the form
\begin{equation*}
  \Delta_\phi f= e^\phi\,\nabla\cdot\(\nabla f\,e^{-\phi}\)
  = \nablaphi \cdot \nabla f =
  \Delta f - \nabla \phi \cdot \nabla f\,.
\end{equation*}
\smallskip

\noindent$\rhd$ By definition of $D^su$ and using integration by
parts, we have
\begin{equation*}
  \begin{array}{rl}
    (-\Delta_S u,u)\kern-8pt
    & =2\int_{\R^d}|D^su|^2\,e^{-\phi}\dd
      x=\frac12\sum_{i,j=1}^d
      \int_{\R^d}\(\partial_iu_j+\partial_ju_i\)^2e^{-\phi}\dd
      x\\
    &=-\frac12\sum_{i,j=1}^d\int_{\R^d}u_j\,
      \partial_i\((\partial_iu_j+\partial_ju_i)\,e^{-\phi}\)\dd
      x
      -\frac12\sum_{i,j=1}^d\int_{\R^d}u_i\,
      \partial_j\((\partial_iu_j+\partial_ju_i)\,
      e^{-\phi}\)\dd x\\
    &=-\sum_{i,j=1}^d\int_{\R^d}u_j\,
      \partial_i\((\partial_iu_j+\partial_ju_i)\,
      e^{-\phi}\)\dd x\\
    &=-\sum_{i=1}^d\int_{\R^d}\big(u_i\,\Delta
      u_i+u_i\,\partial_{ij}u_j\big)\,e^{-\phi}\dd
      x
      +\sum_{i=1}^d\int_{\R^d}u_i\,
      \big((\nabla\phi\cdot\nabla)\,u_i
      +2\,(D^su\,\nabla\phi)_i\big)\,
      e^{-\phi}\dd x\\
    &=-\int_{\R^d}u\cdot\big(\Delta
      u+\nabla(\nabla\cdot
      u)-(\nabla\phi\cdot\nabla)\,u
      -2\,D^su\,\nabla\phi\big)\dd x\,.
  \end{array}
\end{equation*}
so that $-\Delta_S$ is given, for an arbitrary vector field
$u\in\ccc_c^\infty(\R^d;\R^d)$, by
\begin{equation*}
  -\Delta_S\,u=-\,D^s_\phi\cdot D^su=
  -\big(\Delta u+\nabla(\nabla\cdot u)
  -(\nabla\phi\cdot\nabla)\,u-2\,D^su\,\nabla\phi\big)\,.
\end{equation*}

%%%%%%%%%%%%%%%%%%%%%%%%%%%%%%%%%%%%%%%%%%%%%%%%%%%%%%%%%%%%%
\subsection{Gaussian measure}\label{Appendix:B5}

In the normalized centred Gaussian case
$\phi(x)=\frac12\,|x|^2+\frac d2\,\ln(2\pi)$ corresponding
to~\eqref{Gaussian}, the basic constants are \hbox{$\CP=1$}
(which is the optimal constant in the Gaussian Poincar\'e
inequality), either $C_\phi=1+4\,d$ and
$C_\phi'=4\,\sqrt{d\,(1+4\,d)}$ if $d\ge2$, or $C_\phi=8$ and
$C_\phi'=8\,\sqrt2$ if $d=1$, as a limit case.

Let $u(x)=(1-x_2^2,x_1\,x_2,0,\ldots0)^\perp$. By elementary
computations, we find that
\begin{equation*}
  Du=
  \begin{pmatrix}
    \begin{array}{cc}
      0&-\,2\,x_2\\
      x_2&x_1\end{array}&{\mathbf0}\\
    {\mathbf0}&{\mathbf0}
  \end{pmatrix},\quad
  D^su=
  \begin{pmatrix}
    \begin{array}{cc}0&-\,\frac12\,x_2\\
      -\,\frac12\,x_2&x_1\end{array}&{\mathbf0}\\
    {\mathbf0}&{\mathbf0}
  \end{pmatrix},\quad
  D^au=
  \begin{pmatrix}
    \begin{array}{cc}0&-\,\frac32\,x_2\\
      \frac32\,x_2&0\end{array}&{\mathbf0}\\
    {\mathbf0}&{\mathbf0}
  \end{pmatrix}
\end{equation*}
where $\mathbf0$ denotes $2\times(d-2)$, $(d-2)\times2$, and
$(d-2)\times(d-2)$ null matrices. After integration against the
normalized centred Gaussian measure, we have
\begin{equation*}
  \seq u=0,\quad\|u\|^2=3\,,\quad
  \P(u)=0\,,\quad\fP (Du)=0=\seq{D^au}\,,\quad
  \|D^su\|^2=\frac32\,,\quad\|D^au\|^2=\frac92\,,
  \quad\|Du\|^2=6\,.
\end{equation*}
This proves that $\CK=4$ in
$\|\nablaD u-\fP (Du)\|^2 \leq \CK\,\|D^s u\|^2$ and $\CPK=2$ in
$\|u-\seq{u} - \P (u)\|^2 \leq \CPK\,\|D^s u\|^2$ are both
optimal.

%%%%%%%%%%%%%%%%%%%%%%%%%%%%%%%%%%%%%%%%%%%%%%%%%%%%%%%%%%%%%
\subsection{Estimates for the
  \texorpdfstring{$D(\Lambda)$}{Dlambda}-Toolbox and
  consequences}
\label{Sec:Dlambda}

Here we give some details on the computation of $\CB$ in the
proof of Proposition~\ref{prop:domain} in
Section~\ref{Sec:Toolboxes}. Let $A=\big\|\wgt\,\nabla f\big\|$,
$B=\big\|\wgt^2\,f\big\|$, $Z=\|\xi\|$ and let
${c'}^2=C_\phi'\ge C_\phi=c^2$. Inequalities~\eqref{DVDu}
and~\eqref{DVDubis} amount to
\begin{equation*}
  B\leq\tfrac43\,c\,A+\tfrac53\,{c'}^2\,Z\quad\mbox{and} \quad
  A^2\leq B\,Z+\tfrac1{2\,c}\,A\,B+\tfrac{{c'}^2}{2\,c}\,A\,Z\,.
\end{equation*}
Taking the equality case in the first inequality, we find that
\begin{equation*}
  A^2-4\,A\,\big(c+{c'}^2/c\big)\,Z-5\,{c'}^2\,Z^2\le0
\end{equation*}
which means that
\begin{equation*}
  \textstyle A\le\(2\,\sigma+\sqrt{4\,\sigma^2+5}\,\)c'\,Z
  \quad\mbox{with}\quad\sigma=\frac c{c'}+\frac{c'}c\ge2\,.
\end{equation*}
On $[2,+\infty)$, the function
$\sigma\mapsto\sqrt{4\,\sigma^2+5}/\sigma$ is monotone
non-increasing, so that
$\sqrt{4\,\sigma^2+5}\le\frac12\,\sqrt{21}\,\sigma$. Using the
monotonicity of $c\mapsto c\,\sigma$ and $c\le c'$, we also have
$c\,\sigma\le2\,c'$. As a consequence, we have
\begin{align*}
  &\textstyle A\le\(2+\tfrac12\,\sqrt{21}\,\)c\,\sigma\,
    \frac{c'}c\,Z\le\(4+\sqrt{21}\,\)\frac{{c'}^2}c\,Z
    \le 9\,\frac{{c'}^2}c\,Z\,,\\
  &\textstyle B
    \le\frac13\(4\,c\(2\,\sigma
    +\sqrt{4\,\sigma^2+5}\,\)+5\,c'\)c'\,Z\le\(7+4\,\sqrt{7/3}\,\)
    \le14\,{c'}^2\,Z\,,
\end{align*}
that is, the bounds~\eqref{Est:4}. Moreover, from
$\big\|D^2f\big\|^2\le\tfrac
Cd+\tfrac12\(d+\sqrt{d^2+4\,C}\,\)\|\xi\|^2$ with
$C=\tfrac{277}8\(C_\phi-1\){C_\phi'}^2$, we deduce that
\begin{equation}
  \label{Eq:CB}
  \CB=81\,\tfrac{{C_\phi'}^2}{C_\phi}
  +196\,{C_\phi'}^2+\tfrac Cd+\tfrac12\(d+\sqrt{d^2+4\,C}\,\).
\end{equation}

%%%%%%%%%%%%%%%%%%%%%%%%%%%%%%%%%%%%%%%%%%%%%%%%%%%%%%%%%%%%%%
\subsection{Estimates on various constants}
\label{Sec:Constants}

The constants $\Cphi$ and $\Cphi'$ appear in~\eqref{eq:cphi} as a
consequence of~\eqref{hyp:regularity} while the Poincar\'e
constant $\CP$ follows from
Assumption~\eqref{hyp:poincarebasique}. According to~\eqref{CPL},
the constant in the Poincar\'e-Lions
inequality~\eqref{eq:poincarelions} is given with~$\CB$ as
in~\eqref{Eq:CB} by
$\CPL=(1+\CP)^2\,\big(1+\Cphi'\,\sqrt{\CB}/4\big)^2$. From
Proposition~\ref{prop:poincares}, we know that the strong
Poincar\'e inequality~\eqref{eq:strongpoincare} holds for some
$\CSP\leq\Cphi\,(1+\CP)$. As for the other constants in
Theorems~\ref{theo:KPK} and~\ref{theo:PKPK}, we learn from the
proofs in Sections~\ref{Sec:Prf1} and~\ref{Sec:Prf2} that
\begin{align*}
  &\CK\leq1+4\,\CPL\,,\quad\CPK\leq\CP\,\CK\,,
    \quad\CSPK\leq\CSP(\CK+3\,\Cphi\,\CPK)\,,\\
  &\CPK'\leq\CPK+2\,\CRV\,\CSPK\quad\mbox{and}\quad
    \CK'\leq\CK(1+2\,\CRD\,\CSP)\,.
\end{align*}
In Section~\ref{Sec:6.1}, Lemma~\ref{LemPL-1}, using
$\Cphi'':=C_\eps$ as in~\eqref{hyp:regularity} with
$\eps=1/(2\,\sqrt{\CB})$, the constants
in~\eqref{eq:poincarelions-1} are
\begin{equation*}
  \CLPL \leq 4\,(1+\CP)^2\,(1+\Cphi'')^2\quad\mbox{and}\quad
  \CRPL\leq
  \(1+\tfrac14\,\Cphi'\,{\textstyle\sqrt{\CB/C_\phi}}\,\)^2\,.
\end{equation*}
Finally, the constants in Theorem~\ref{theo:KPK0} are given by
\begin{equation*}
  \CKZ=1+4\,\CLPL\,,\quad\CPKZ \leq \CPL\,(1+4\,\CLPL)
  \quad\mbox{and}\quad\CPKZ' \leq \CPKZ(1+2\,\CRVZ\,\Cphi).
\end{equation*}

%%%%%%%%%%%%%%%%%%%%%%%%%%%%%%%%%%%%%%%%%%%%%%%%%%%%%%%%%%%%%%%%
%%%%%%%%%%%%%%%%%%%%%%%%%%%%%%%%%%%%%%%%%%%%%%%%%%%%%%%%%%%%%%%%
\section*{Acknowledgments} \noindent\quad \small
This work has been partially supported by the Projects EFI (K.C.,
J.D., ANR-17-CE40-0030) and Kibord (K.C., J.D., S.M.,
ANR-13-BS01-0004) of the French National Research Agency
(ANR). C.M.~and S.M.~acknowledge partial funding by the ERC
grants MATKIT 2011-2016 and MAFRAN 2017-2022. Moreover C.M.~is
very grateful for the hospitality at Universit\'e Paris-Dauphine.\\
\noindent{\scriptsize\copyright\,2020 by the authors. This paper
  may be reproduced, in its entirety, for non-commercial
  purposes.}
%%%%%%%%%%%%%%%%%%%%%%%%%%%%%%%%%%%%%%%%%%%%%%%%%%%%%%%%%%%%%%%%
%\bibliographystyle{acm}
%\bibliography{bib-korn}

%%%%%%%%%%%%%%%%%%%%%%%%%%%%%%%%%%%%%%%%%%%%%%%%%%%%%%%%%%%%%%%%
% To include the Table of Contents, please uncomment the next
% line
%\newpage\setcounter{tocdepth}2\tableofcontents\vspace*{-0.5cm}
%%%%%%%%%%%%%%%%%%%%%%%%%%%%%%%%%%%%%%%%%%%%%%%%%%%%%%%%%%%%%%%%
\bigskip\begin{center}\rule{2cm}{0.5pt}\end{center}
\end{document}